\documentclass[a4paper,10pt]{amsart}

\usepackage{amsmath}
\usepackage{amsthm}
\usepackage{amsfonts}
\usepackage{amssymb}
\usepackage[all]{xy}

\usepackage{amssymb,amsmath,amsfonts,amscd,euscript}

\newcommand{\nc}{\newcommand}

\numberwithin{equation}{section}
\newtheorem{thm}{Theorem}[section]
\newtheorem{prop}[thm]{Proposition}
\newtheorem{lem}[thm]{Lemma}
\newtheorem{cor}[thm]{Corollary}
\theoremstyle{remark}
\newtheorem{rem}[thm]{Remark}

\newtheorem{example}[thm]{Example}
\newtheorem{dfn}[thm]{Definition}

\nc{\gl}{\mathfrak{gl}}
\nc{\GL}{\mathfrak{GL}}
\nc{\g}{\mathfrak{g}}
\nc{\gh}{\widehat\g}
\nc{\h}{\mathfrak{h}}
\nc{\la}{\lambda}
\nc{\al}{\alpha }
\nc{\ve}{\varepsilon }
\nc{\om}{\omega }

\nc{\ta}{\theta}
\nc{\veps}{\varepsilon}
\nc{\ch}{{\mathop {\rm ch}}}
\nc{\Tr}{{\mathop {\rm Tr}\,}}
\nc{\Id}{{\mathop {\rm Id}}}
\nc{\ad}{{\mathop {\rm ad}}}
\nc{\bra}{\langle}
\nc{\ket}{\rangle}
\nc{\x}{{\bf x}}
\nc{\bs}{{\bf s}}
\nc{\bp}{{\bf p}}
\nc{\bc}{{\bf c}}
\nc{\pa}{\partial}
\nc{\ld}{\ldots}
\nc{\cd}{\cdots}
\nc{\hk}{\hookrightarrow}
\nc{\T}{\otimes}
\newcommand{\bea}{\begin{equation}}
\newcommand{\ena}{\end{equation}}
\nc{\gr}{\mathrm{gr}}
\nc{\ov}{\overline}

\nc{\cO}{\mathcal O}
\nc{\cF}{\mathcal F}
\nc{\cL}{\mathcal L}
\nc{\msl}{\mathfrak{sl}}
\nc{\mgl}{\mathfrak{gl}}
\nc{\U}{\mathrm U}
\nc{\V}{\EuScript V}
\nc{\bH}{\EuScript H}
\nc{\Res}{\mathrm{Res\ }}

\newcommand{\bC}{{\mathbb C}}
\newcommand{\bZ}{{\mathbb Z}}

\newcommand{\bG}{{\mathbb G}}

\newcommand{\fb}{{\mathfrak b}}

\newcommand{\fn}{{\mathfrak n}}

\newcommand{\Fl}{\EuScript{F}}
\newcommand{\Hom}{\mathrm{Hom}}

\newcommand{\bd}{{\bf d}}
\newcommand{\bff}{{\bf f}}
\newcommand{\bg}{{\bf g}}
\newcommand{\be}{{\bf e}}

\begin{document}

\title[QGDF]{Quiver Grassmannians and degenerate flag varieties}
\author{Giovanni Cerulli Irelli, Evgeny Feigin, Markus Reineke}
\address{Giovanni Cerulli Irelli:\newline
Sapienza Universit\`a di Roma. Piazzale Aldo Moro 5, 00185, Rome (ITALY)}
\email{cerulli.math@googlemail.com}
\address{Evgeny Feigin:\newline
Department of Mathematics,\newline National Research University Higher School of Economics,\newline
20 Myasnitskaya st, 101000, Moscow, Russia\newline
{\it and }\newline
Tamm Department of Theoretical Physics,
Lebedev Physics Institute
}
\email{evgfeig@gmail.com}
\address{Markus Reineke:\newline 
Fachbereich C - Mathematik, Bergische Universit\"at Wuppertal, D - 42097 Wuppertal}
\email{reineke@math.uni-wuppertal.de}



\begin{abstract}
Quiver Grassmannians are varieties parametrizing subrepresentations of a
quiver representation. It is observed that certain quiver
Grassmannians for type A quivers are isomorphic to the degenerate flag
varieties investigated earlier by the second named author. This leads to
the consideration of a class of Grassmannians of subrepresentations of the
direct sum of a projective and an injective representation of a Dynkin
quiver. It is proven that these are (typically singular)
irreducible normal local complete intersection varieties, which
admit a group action with finitely many orbits, and a cellular
decomposition. For type A quivers explicit formulas for the Euler characteristic 
(the median Genocchi numbers) and the Poincar\'e polynomials are derived.
\end{abstract}
\maketitle
\section{Introduction}

\subsection{Motivation}
Quiver Grassmannians, which are varieties parametrizing subrepresentations of a quiver representation, first appeared in \cite{CrawleyTree,SchofieldGeneric} in relation to questions on generic properties of quiver representations. It was observed in \cite{CC} that these varieties play an important role in cluster algebra theory \cite{FZI}; namely, the cluster variables can be described in terms of the Euler characteristic of quiver Grassmannians. Subsequently, specific classes of quiver Grassmannians (for example, varieties of subrepresentations of exceptional quiver representations) were studied by several authors, with the principal aim of computing their Euler characteristic explicitly; see for example \cite{CR,string,CEsp,CeDEsp}. In recent papers authors noticed that also the Poincar\'e polynomials of quiver Grassmannians play an important role in the study of quantum cluster algebras \cite{Qin, BZ}\\[1ex]
This paper originated in the observation that a certain quiver Grassmannians can be identified with the $\mathfrak{sl}_n$-degenerate flag variety of \cite{F1,F2,FF}. This leads to the consideration of a class of Grassmannians of subrepresentations of the direct sum of a projective and an injective representation of a Dynkin quiver. It turns out that this class of varieties enjoys many of the favourable properties of quiver Grassmannians for exceptional representations. More precisely, they turn out to be (typically singular) irreducible normal local complete intersection varieties which admit a group action with finitely many orbits and a cellular decomposition. The proofs of the basic geometric properties are based on generalizations of the techniques of \cite{ReFM}, where the case of Grassmannians of subrepresentations of injective quiver representations is treated.

\subsection{Main results} Let $Q$ be a quiver with set of vertices $Q_0$ of cardinality $n$ and finite set of arrows $Q_1$.
For a representation $M$ of $Q$ we denote by $M_i$ the space in $M$ attached to the $i$-th vertex, and by $M_\alpha:M_i\rightarrow M_j$ the linear map attached to an arrow $\alpha:i\rightarrow j$.
We also denote by $\langle\_,\_\rangle$ the Euler form on $\bZ Q_0$.
Given a dimension vector $\be=(e_1,\dots,e_n)\in \bZ_{\ge 0} Q_0$ and a representation $M$ of $Q$, the quiver Grassmannian
${\rm Gr}_\be(M)\subset \prod_{i=1}^n {\rm Gr}_{e_i}(M_i)$ is the subvariety of collections of subspaces
$V_i\subset M_i$ subject to the conditions $M_\al V_i\subset V_j$ for all $\al:i\to j\in Q_1$.
In this paper we study a certain class of quiver Grassmannians for Dynkin quivers $Q$.
Before describing this class, we first consider the following example.

Let $Q$ be an equioriented
quiver of type $A_n$ with vertices $i=1,\dots,n$ and arrows $i\to i+1$, and let
$\bC Q$ be the path algebra of $Q$. Then the quiver Grassmannian
${\rm Gr}_{{\bf dim}\, \bC Q}(\bC Q\oplus \bC Q^*)$ is isomorphic to the
complete degenerate flag variety $\Fl^a_{n+1}$ for $G=SL_{n+1}$. Let us recall the definition
(see \cite{F1},\cite{F2}).
Let $W$ be an $(n+1)$-dimensional vector space with basis $w_1,\dots,w_{n+1}$. Let $pr_k:W\to W$ for $k=1,\dots,n+1$ be the projection operators $pr_k (\sum_{i=1}^{n+1} c_iw_i)=
\sum_{i\ne k} c_iw_i$. Then the degenerate flag variety consists of collections $(V_1,\dots,V_n)$
with $V_i\subset W$ and $\dim V_i=i$, subject to the conditions $pr_{k+1} V_k\subset V_{k+1}$, $k=1,\dots,n-1$.
These varieties are irreducible singular algebraic varieties enjoying  many nice properties.
In particular, they are flat degenerations of classical flag varieties $SL_{n+1}/B$.
Now let us consider the representation $M$ of $Q$ such that $M_i=W$ and the maps $M_i\to M_{i+1}$ are given by $pr_{i+1}$.
For example, for $n=3$ $M$ has the following coefficient quiver
$$
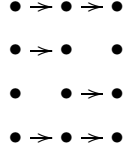
\xymatrix@R=6pt@C=8pt{
\bullet\ar[r]&\bullet\ar[r]&\bullet\\
\bullet\ar[r]&\bullet&\bullet \\
\bullet&\bullet\ar[r]&\bullet \\
\bullet\ar[r]&\bullet\ar[r]&\bullet \\
}
$$
where each dot represents basis vectors $w_1,w_2,w_3,w_4$ from bottom to top and arrows represent maps.
Note that $M$ is isomorphic to $\bC Q\oplus \bC Q^*$ as a representation of $Q$ and moreover, we have
\begin{equation}\label{iso}
\Fl^a_{n+1}\simeq {\rm Gr}_{{\bf dim}\, \bC Q}(\bC Q\oplus \bC Q^*).
\end{equation}

Now let $Q$ be a Dynkin quiver. Recall that the path algebra $\bC Q$ (resp. its linear dual $\bC Q^*$) is isomorphic
as a representation of $Q$ to the direct sum of all indecomposable projective representations (resp. all indecomposable injective
representations). Motivated by the isomorphism \eqref{iso} we consider the quiver Grassmannians
${\rm Gr}_{{\bf dim}\, P} (P\oplus I)$, where $P$ and $I$ are projective resp. injective representations of $Q$. (We note that
some of our results are valid for more general Grassmannians and we discuss it in the main body of the paper.
However, in the introduction we restrict ourselves to the above mentioned class of varieties).
We use the isomorphism \eqref{iso} in two different ways. On one hand, we generalize and expand the results
about $\Fl^a_{n+1}$ to the case of the above quiver Grassmannians. On the other hand, we use general results and constructions
from the theory of quiver representations to understand better the structure of $\Fl^a_{n+1}$.
Our first theorem is as follows:
\begin{thm}
The variety ${\rm Gr}_{{\bf dim}\, P} (P\oplus I)$ is of dimension $\langle {\bf dim}\, P,{\bf dim}\, I\rangle$ and irreducible.
It is a normal local complete intersection variety.
\end{thm}

Our next goal is to construct cellular decompositions of the quiver Grassmannians and to compute their
Poincar\'e polynomials. Let us consider the following stratification of ${\rm Gr}_{{\bf dim}\, P} (P\oplus I)$.
For a point $N$ we set $N_I=N\cap I$, $N_P=\pi N$, where $\pi:P\oplus I\to P$ is the projection. Then
for a dimension vector $\bff\in \bZ_{\ge 0} Q_0$ we set
\[
\mathcal{S}_\bff=\{N\in {\rm Gr}_{{\bf dim}\, P} (P\oplus I):\ {\bf dim}\, N_I=\bff, {\bf dim}\, N_P={\bf dim}\, P -\bff\}.
\]
We have natural surjective maps 
$\zeta_\bff:\mathcal{S}_\bff\to {\rm Gr}_{\bff}(I)\times {\rm Gr}_{{\bf dim}\, P-\bff}(P)$.
\begin{thm}\label{Poincare}
The map $\zeta_\bff$ is a vector bundle. The fiber over a point $(N_P,N_I)$ is isomorphic to 
$Hom_Q(N_P,I/N_I)$ which has dimension $\langle {\bf\dim} P-\bff, {\bf\dim}I-\bff\rangle$.
\end{thm}

Using this theorem, we construct a cellular decomposition for each stratum $X_\bff$ and 
thus for the whole variety $X$ as
well. Moreover, since the Poincar\'e polynomials of ${\rm Gr}_{\bff}(I)$ and of ${\rm Gr}_{{\bf dim}\, P-\bff}(P)$ can be easily
computed, we arrive at a formula for the Poincar\'e polynomial (and thus for the Euler characteristic) of $X$.
Recall (see \cite{F1}) that the Euler characteristic of the variety $\Fl^a_{n+1}$ is given by the
normalized median Genocchi number $h_{n+1}$ (see \cite{De,Du,DR,DZ,Vien}). 
Using Theorem \ref{Poincare} we obtain an explicit formula
for $h_{n+1}$ in terms of binomial coefficients. Moreover, we give a formula for the Poincar\'e polynomial of
$\Fl^a_{n+1}$,  providing a natural $q$-version of $h_{n+1}$.

Finally we study the action of the group of automorphisms $Aut(P\oplus I)$ on the quiver Grassmannians.
Let $G\subset Aut(P\oplus I)$ be the group
$$
G=\left[\begin{array}{cc}{\rm Aut}_Q(P)&0\\ {\rm Hom}_Q(P,I)&{\rm Aut}_Q(I)\end{array}\right]
$$
(we note that $G$ coincides with the whole group of automorphisms unless $Q$ is of type $A_n$).
We prove the following theorem:
\begin{thm}
The group $G$ acts on ${\rm Gr}_{{\bf dim} P}(P\oplus I)$ with finitely many orbits, parametrized by pairs of isomorphism classes $([Q_P],[N_I])$ such that $Q_P$ is a quotient of $P$, $N_I$ is a subrepresentation of $I$, and
${\bf dim}\, Q_P={\bf dim}\, N_I$. Moreover, if $Q$ is equioriented of type $A_n$, then the orbits are cells parametrized by
torus fixed points.
\end{thm}

\subsection{Outline of the paper} Our paper is organized as follows:\\
In Section \ref{GF} we recall general facts about quiver Grassmannians and degenerate flag varieties.\\
In Section \ref{Sec:WellBehaved} we prove that the quiver Grassmannians ${\rm Gr}_{{\bf dim}} (P\oplus I)$
are locally complete intersections and that they are flat degenerations of the Grassmannians in
exceptional representations.\\
In Section \ref{CGA} we study the action of the automorphism group on ${\rm Gr}_{{\bf dim}} (P\oplus I)$,
describe the orbits and prove the normality of ${\rm Gr}_{{\bf dim}} (P\oplus I)$.\\
In Section \ref{1dtorus} we construct a one-dimensional torus action on our quiver Grassmannians
such that the attracting sets form a cellular decomposition.\\
Sections \ref{PPEC} and \ref{CGAA} are devoted to the case of the equioriented quiver of type $A$.
In Section \ref{PPEC} we compute the Poincar\'e polynomials of  ${\rm Gr}_{{\bf dim}} (P\oplus I)$
and derive several new  formulas for the Euler characteristic -- the normalized median
Genocchi numbers.  In Section \ref{CGAA} we prove that the orbits studied in Section \ref{CGA}
are cells coinciding with the attracting cells constructed in Section \ref{1dtorus}.
We also describe the connection with the degenerate group $SL_{n+1}^a$.

\section{General facts on quiver Grassmannians and degenerate flag varieties}\label{GF}

\subsection{General facts on quivers} Let $Q$ be a finite quiver with finite set of vertices $Q_0$ and finite set of arrows $Q_1$; arrows will be written as $(\alpha:i\rightarrow j)\in Q_1$ for $i,j\in Q_0$. We assume $Q$ to be without oriented cycles. Denote by ${\bf Z}Q_0$ the free abelian group generated by $Q_0$, and by ${\bf N}Q_0$ the subsemigroup of dimension vectors ${\bf d}=(d_i)_{i\in Q_0}$ for $Q$. Let $\langle\_,\_\rangle$ be the Euler
form on ${\bf Z}Q_0$, defined by
$$\langle{\bf d},{\bf e}\rangle=\sum_{i\in Q_0}d_ie_i-\sum_{(\alpha:i\rightarrow j)\in Q_1}d_ie_j.$$
We consider finite dimensional representations $M$ of $Q$ over the complex numbers, viewed either as finite dimensional left modules over the path algebra ${\bf C}Q$ of $Q$, or as tuples $M=((M_i)_{i\in Q_0},(M_\alpha:M_i\rightarrow M_j)_{(\alpha:i\rightarrow j)\in Q_1})$ consisting of finite dimensional complex vector spaces $M_i$ and linear maps $M_\alpha$. The category ${\rm rep}(Q)$ of all such representations is hereditary (that is, ${\rm Ext}_Q^{\geq 2}(\_,\_)=0$). Its Grothendieck group $K({\rm rep}(Q))$ is isomorphic to ${\bf Z}Q_0$ by identifying the class of a representation $M$ with its dimension vector 
${\bf dim}\,M=(\dim M_i)_{i\in Q_0}\in{\bf Z}Q_0$. The Euler form defined above then identifies with the homological Euler form, that is,
$$\dim{\rm Hom}_Q(M,N)-\dim{\rm Ext}_Q^1(M,N)=\langle{\bf dim}\,M,{\bf dim}\,N\rangle$$
for all representations $M$ and $N$.\\[1ex]
Associated to a vertex $i\in Q_0$, we have the simple representation $S_i$ of $Q$ with 
$({\bf dim}\,S_i)_j=\delta_{i,j}$ (the Kronecker delta), the projective indecomposable $P_i$, and the indecomposable injective $I_i$. The latter are determined as the projective cover (resp. injective envelope) of $S_i$; more explicitly, $(P_i)_j$ is the space generated by all paths from $i$ to $j$, and the linear dual of $(I_i)_j$ is the space generated by all paths from $j$ to $i$.\\[1ex]
Given a dimension vector ${\bf d}\in{\bf N}Q_0$, we fix complex vector spaces $M_i$ of dimension $d_i$ for all $i\in Q_0$. We consider the affine space
$$R_{\bf d}(Q)=\bigoplus_{(\alpha:i\rightarrow j)}{\rm Hom}_{\bf C}(M_i,M_j);$$
its points canonically parametrize representations of $Q$ of dimension vector ${\bf d}$. The reductive algebraic group $G_{\bf d}=\prod_{i\in Q_0}{\rm GL}(M_i)$ acts naturally on $R_{\bf d}(Q)$ via base change
$$(g_i)_i\cdot(M_\alpha)_\alpha=(g_jM_\alpha g_i^{-1})_{(\alpha:i\rightarrow j)},$$
such that the orbits $\mathcal{O}_M$ for this action naturally correspond to the isomorphism classes $[M]$ of representations of $Q$ of dimension vector ${\bf d}$. Note that $\dim G_{\bf d}-\dim R_{\bf d}(Q)=\langle {\bf d},{\bf d}\rangle$. The stabilizer under $G_{\bf d}$ of a point $M\in R_{\bf d}(Q)$ is isomorphic to the 
automorphism group ${\rm Aut}_Q(M)$ of the corresponding representation, 
which (being open in the endomorphism space ${\rm End}_Q(M)$) is a connected algebraic group of dimension $\dim{\rm End}_Q(M)$. In particular, we get the following formulas:
\begin{equation}\label{Eq:OrbitDim}
\dim\mathcal{O}_M=\dim G_d-\dim{\rm End}_Q(M),\;\;\;{\rm codim}_{R_{\bf d}}\mathcal{O}_M=\dim{\rm Ext}_Q^1(M,M).
\end{equation}


\subsection{Basic facts on quiver Grassmannians} 
The constructions and results in this section follow \cite{CR},\cite{SchofieldGeneric}. Additionally to the above, fix another dimension vector ${\bf e}$ such that
${\bf e}\leq{\bf d}$ componentwise, and define the $Q_0$-graded Grassmannian
${\rm Gr}_{\bf e}({\bf d})=\prod_{i\in Q_0}{\rm Gr}_{e_i}(M_i)$,
which is a projective homogeneous space for $G_{\bf d}$ of dimension
$\sum_{i\in Q_0}e_i(d_i-e_i)$, namely ${\rm Gr}_{\bf e}({\bf d})\simeq G_{\bf d}/P_{\bf e}$
for a maximal parabolic $P_{\bf e}\subset G_{\bf d}$. We define
${\rm Gr}_{\bf e}^Q({\bf d})$, the universal Grassmannian of ${\bf e}$-dimensional
subrepresentations of ${\bf d}$-dimensional representations of $Q$ as the closed subvariety of
${\rm Gr}_{\bf e}({\bf d})\times R_{\bf d}(Q)$ consisting of tuples
$((U_i\subset M_i)_{i\in Q_0},(M_\alpha)_{\alpha\in Q_1})$ such that $M_\alpha(U_i)\subset U_j$
for all arrows $(\alpha:i\rightarrow j)\in Q_1$. The group $G_{\bf d}$ acts on
${\rm Gr}_{\bf e}^Q({\bf d})$ diagonally, such that the projections
$p_1:{\rm Gr}_{\bf e}^Q({\bf d})\rightarrow {\rm Gr}_{\bf e}({\bf d})$
and $p_2:{\rm Gr}_{\bf e}^Q({\bf d})\rightarrow R_{\bf d}(Q)$ are
$G_{\bf d}$-equivariant. In fact, the projection $p_1$ identifies ${\rm Gr}_{\bf e}^Q({\bf d})$
as the total space of a homogeneous bundle over ${\rm Gr}_{\bf e}({\bf d})$ of rank
$$\sum_{(\alpha:i\rightarrow j)\in Q_1}(d_id_j+e_ie_j-e_id_j).$$
Indeed, for a point $(U_i)_{i=1}^{\#Q_0}$ in ${\rm Gr}_{\bf e}({\bf d})$, we can choose complements 
$M_i=U_i\oplus V_i$ and identify the fiber of $p_1$ over $(U_i)_{i=1}^{\#Q_0}$ with
$$\left(\left[
\begin{array}{cc}{\rm Hom}_Q(U_i,U_j)&{\rm Hom}_Q(V_i,U_j)\\
0&{\rm Hom}_Q(V_i,V_j)
\end{array}
\right]
\subset{\rm Hom}_Q(M_i,M_j)\right)_{(\alpha:i\rightarrow j)}\subset R_{\bf d}(Q).$$
In particular, the universal Grassmannian ${\rm Gr}_{\bf e}^Q({\bf d})$ is smooth and
irreducible of dimension
$$\dim{\rm Gr}_{\bf e}^Q({\bf d})=\langle {\bf e},{\bf d}-{\bf e}\rangle+\dim R_{\bf d}(Q).$$
The projection $p_2$ is proper, thus its image is a closed $G_{\bf d}$-stable subvariety of
$R_{\bf d}$, consisting of representations admitting a subrepresentation of dimension vector
${\bf e}$.\\[1ex]
We define the quiver Grassmannian ${\rm Gr}_{\bf e}(M)=p_2^{-1}(M)$ as the fibre of $p_2$ over a point
$M\in R_{\bf d}(Q)$; by definition, it parametrizes (more precisely, its closed points parametrize)
${\bf e}$-dimensional subrepresentations of the representation $M$.
\begin{rem} Note that we have to view ${\rm Gr}_{\bf e}(M)$ as a scheme; in particular, it might be non-reduced.  For example, if $Q$ is the Kronecker quiver, ${\bf e}$ the isotropic root, and $M$ is a regular indecomposable of dimension vector $2{\bf e}$, the quiver Grassmannian is ${\rm Spec}$ of the ring of dual numbers.
\end{rem}

Recall that a representation $M$ is called exceptional if ${\rm Ext}_Q^1(M,M)=0$; thus, in view of \eqref{Eq:OrbitDim}, its orbit in
$R_{\bf d}(Q)$ is open and dense.

\begin{prop}\label{qgex} Let $M$ be an exceptional ${\bf d}$-dimensional representation of $Q$. Then
 ${\rm Gr}_{\bf e}(M)$ is non-empty if ${\rm Ext}^1_Q(N,L)$ vanishes
for generic
$N$ of dimension vector ${\bf e}$ and generic $L$ of dimension vector
${\bf d}-{\bf e}$. In this case, ${\rm Gr}_{\bf e}(M)$ is smooth of dimension
$\langle {\bf e},{\bf d}-{\bf e}\rangle$, and for all ${\bf d}$-dimensional representations $N$, every irreducible component of ${\rm Gr}_{\bf e}(N)$ has at least dimension $\langle {\bf e},{\bf d}-{\bf e}\rangle$.
\end{prop}
\begin{proof} The criterion for non-emptyness follows from \cite[Theorem 3.3]{SchofieldGeneric}.
If ${\rm Gr}_{\bf e}(M)$ is non-empty, $p_2$ is surjective with
${\rm Gr}_{\bf e}(M)$ as its generic fibre. In particular, ${\rm Gr}_{\bf e}(M)$ is smooth of
dimension $\langle {\bf e},{\bf d}-{\bf e}\rangle$. For all other fibres, we obtain at least the desired estimate on dimensions of their irreducible components \cite[Ch. II, Exercise 3.22 (b)]{Hartshorne}.
\end{proof}
We conclude this section by pointing out an useful isomorphism: let $U$ be a point of ${\rm Gr}_{\bf e}(M)$ and let $T_U({\rm Gr}_{\bf e}(M))$ denote the tangent space of ${\rm Gr}_{\bf e}(M)$  at $U$. As shown in \cite{SchofieldGeneric, CR} we have the following scheme--theoretic description of the tangent space:

\begin{lem}\label{Lem:TangentSpace}
For $U\in{\rm Gr}_{\bf e}(M)$, we have $T_U({\rm Gr}_{\bf e}(M))\simeq{\rm Hom}_Q(U,M/U)$.
\end{lem}

\subsection{Quotient construction of (universal) quiver Grassmannians and a stratification}\label{qcqg}
We follow \cite[Section 3.2]{ReFM}. Additionally to the choices before,
fix vector spaces $N_i$ of dimension $e_i$ for $i\in Q_0$. We consider the universal variety
${\rm Hom}_Q({\bf e},{\bf d})$ of homomorphisms from an ${\bf e}$-dimensional to a
${\bf d}$-dimensional representation; explicitly, ${\rm Hom}_Q({\bf e},{\bf d})$ is the set
of triples
$$((N_\alpha)_{\alpha\in Q_1},(f_i:N_i\rightarrow M_i)_{i\in Q_0},
(M_\alpha)_{\alpha\in Q_1})\in R_{\bf e}\times\prod_{i\in Q_0}{\rm Hom}(N_i,M_i)\times R_{\bf d}(Q)$$
such that $f_jN_\alpha=M_\alpha f_i$ for all $(\alpha:i\rightarrow j)\in Q_1$.
This is an affine variety defined by quadratic relations, namely by the vanishing of the
individual entries of the matrices $f_jN_\alpha-M_\alpha f_i$, on which
$G_{\bf e}\times G_{\bf d}$ acts naturally. On the open subset ${\rm Hom}_Q^0({\bf e},{\bf d})$
where all $f_i:N_i\rightarrow M_i$ are injective maps, the action of $G_{\bf e}$ is free.
By construction, we have an isomorphism
$${\rm Hom}_Q^0({\bf e},{\bf d})/G_e\simeq{\rm Gr}_{\bf e}^Q({\bf d})$$
which associates to the orbit of a triple $((N_\alpha),(f_i),(M_\alpha))$ the pair given by $(f_i(N_i)\subset M_i)_{i\in Q_0},(M_\alpha)_{\alpha\in Q_1})$. Indeed, the maps $N_\alpha$ are uniquely determined in this situation, and they can
be reconstructed algebraically from $(f_i)$ and $(M_\alpha)$ (see \cite[Lemma 3.5]{ReFM}).\\[1ex]
Similarly to ${\rm Gr}_{\bf e}^Q({\bf d})$, we have a projection
$\tilde{p}_2:{\rm Hom}_Q^0({\bf e},{\bf d})\rightarrow R_{\bf d}(Q)$ with fibres $\tilde{p}_2^{-1}(M)={\rm Hom}_Q^0({\bf e},M)$, and we have a local
version of the previous isomorphism
$${\rm Hom}_Q^0({\bf e},M)/G_{\bf e}=\tilde{p}_2^{-1}(M)/G_{\bf e}\simeq{\rm Gr}_{\bf e}(M).$$
Note that the quotient map ${\rm Hom}_Q^0({\bf e},M)\rightarrow {\rm Gr}_{\bf e}(M)$ is locally
trivial, since it is induced by the quotient map
$${\rm Hom}_Q^0({\bf e},{\bf d})=
\prod_{i\in Q_0}{\rm Hom}^0(N_i,M_i)\rightarrow{\rm Gr}_{\bf e}({\bf d}),$$
which can be trivialized over the standard open affine coverings of Grassmannians.\\[1ex]
Let $p$ be the projection from ${\rm Hom}_Q^0({\bf e},M)$ to $R_{\bf e}(Q)$; its fiber over
$N$ is the space ${\rm Hom}^0_Q(N,M)$ of injective maps.
For each isomorphism class $[N]$ of representations of dimension vector ${\bf e}$,
we can consider the subset $\mathcal{S}_{[N]}$ of ${\rm Gr}_{\bf e}(M)$ corresponding
under the previous isomorphism
to $(p^{-1}(\mathcal{O}_N))/G_{\bf e}$. It therefore consists of all subrepresentations
$U\in{\rm Gr}_{\bf e}(M)$ which are isomorphic to $N$.

\begin{lem} Each $\mathcal{S}_{[N]}$ is an irreducible locally closed subset of
${\rm Gr}_{\bf e}(M)$ of dimension $\dim{\rm Hom}_Q(N,M)-\dim{\rm End}_Q(N)$.
\end{lem}
\begin{proof} Irreducibility of $\mathcal{S}_{[N]}$ follows from irreducibility of $\mathcal{O}_N$
by $G_e$-equivariance of $p$. Using the fact that the geometric quotient is closed and
separating on $G_e$-stable subsets, an induction over $\dim\mathcal{O}_N$ proves that all
$\mathcal{S}_{[N]}$ are locally closed. The dimension is calculated as
$$\dim\mathcal{S}_{[N]}=\dim\mathcal{O}_N+\dim{\rm Hom}^0_Q(N,M)-\dim G_e.$$
\end{proof}

\subsection{Degenerate flag varieties}\label{ss1}
In this subsection we recall the definition of the degenerate flag varieties following \cite{F1}, \cite{F2},
\cite{FF}.
Let $W$ be an $n$-dimensional vector space with a basis $w_1,\dots,w_n$. We denote by $pr_k:W\to W$
the projections along $w_k$ to the linear span of the remaining basis vectors, that is,
$pr_k \sum_{i=1}^n c_i w_i= \sum_{i\ne k} c_i w_i$.
\begin{dfn}
The variety $\Fl^a_n$ is the set of collections of subspaces $(V_i\in {\rm Gr}_i(W))_{i=1}^{n-1}$ 
subject to
the conditions $pr_{i+1} V_i\subset V_{i+1}$ for all $i=1,\dots,n-2$.
\end{dfn}

The variety $\Fl^a_n$ is called the complete degenerate flag variety. It enjoys the following properties:
\begin{itemize}
\item $\Fl^a_n$ is a singular irreducible projective algebraic variety of dimension $\binom{n}{2}$.
\item $\Fl^a_n$ is a flat degeneration of the classical complete flag variety $SL_n/B$.
\item $\Fl^a_n$ is a normal local complete intersection variety.
\item $\Fl^a_n$ can be decomposed into a disjoint union of complex cells.
\end{itemize}
We add some comments on the last property. The number of cells (which is equal to the Euler
characteristic of $\Fl^a_n$) is given by the $n$-th normalized median Genocchi number $h_n$ 
(see e.g. \cite{F2}, section $3$ ).
These numbers have several definitions;  here we will use the following one: $h_n$ is the number
of collections $(S_1,\dots,S_{n-1})$, where $S_i\subset \{1,\dots,n\}$ subject to the conditions
\[
\# S_i=i,\ 1\le i\le n-1; \quad S_i\subset S_{i+1}\cup\{i+1\},\ 1\le i\le n-2.
\]
For $n=1,2,3,4,5$ the numbers $h_n$ are equal to $1,2,7,38,295$.

There exists a degeneration $SL_n^a$ of the group $SL_n$ acting on $\Fl^a_n$.
Namely, the degenerate group $SL_n^a$ is the semi-direct product of the Borel subgroup $B$ of $SL_n$
and a normal abelian subgroup $\bG_a^{n(n-1)/2}$, where $\bG_a$ is the additive group of the field.
The simplest way to describe the structure of the semi-direct product is via the Lie
algebra $\msl_n^a$ of $SL_n^a$. Namely, let $\fb\in\msl_n$ be the Borel subalgebra of upper-triangular
matrices and $\fn^-$ be the nilpotent subalgebra of strictly lower-triangular matrices.
Let $(\fn^-)^a$ be the abelian Lie algebra with underlying vector space $\fn^-$. Then
$\fn^-$ carries  natural structure of $\fb$-module induced by the adjoint action on the quotient
$(\fn^-)^a\simeq \msl_n/\fb$. Then $\msl_n^a=\fb\oplus (\fn^-)^a$, where $(\fn^-)^a$ is abelian ideal
and $\fb$ acts on $(\fn^-)^a$ as described above. The group $SL_n^a$ (the Lie group of $\msl_n^a$)
acts on the variety $\Fl^a_n$ with an open $\bG_a^{n(n-1)/2}$-orbit.
We note that in contrast with the classical situation, the group $SL_n^a$ acts on $\Fl^a_n$ 
with an infinite number of orbits.

For partial (parabolic) flag varieties of $SL_n$ there exists a natural generalization of $\Fl^a_n$.
Namely, consider an increasing collection $1\le d_1< \dots < d_s<n$. In what follows we denote such
a collection by $\bd$. Let $\Fl_\bd$ be the classical partial flag variety consisting of the collections
$(V_i)_{i=1}^s$, $V_i\in{\rm Gr}_{d_i}(W)$ such that $V_i\subset V_{i+1}$.  
\begin{dfn}
The degenerate partial variety $\Fl^a_\bd$ is the set of collections of subspaces
$V_i\in {\rm Gr}_{d_i}(W)$ subject to the conditions $pr_{d_i+1}\dots pr_{d_{i+1}} V_i\subset V_{i+1}$ for all $i=1,\dots,s-1$.
\end{dfn}

We still have the following properties
\begin{itemize}
\item $\Fl^a_\bd$ is a singular irreducible projective algebraic variety.
\item $\Fl^a_\bd$ is a flat degeneration of $\Fl_\bd$.
\item $\Fl^a_\bd$ is a normal local complete intersection variety.
\item $\Fl^a_\bd$ is acted upon by the group $SL_n^a$ with an open $\bG_a^{n(n-1)/2}$-orbit.
\end{itemize}

\subsection{Comparison between quiver Grassmannians and degenerate flag varieties}\label{comp}
Let $Q$ be an equioriented quiver of type $A_n$. We order the vertices of $Q$ from $1$ to $n$ in such a way
that the arrows of $Q$ are of the form $i\to i+1$.
Let $P_i$, $I_i$, $i=1,\dots,n$ be the projective and injective representations attached to the $i$-th vertex, respectively.
In particular, ${\bf dim}\, P_i=(0,\dots,0,1,\dots,1)$ with $i-1$ zeros and
${\bf dim}\, I_i=(1,\dots,1,0,\dots,0)$ with $n-i$ zeros.

In what follows we will use the following basis of $P_i$ and $I_i$. Namely, for each
$j=i,\dots,n$ we fix non-zero elements $w_{i,j}\in (P_j)_i$  in such a way that
$w_{i,j}\mapsto w_{i+1,j}$. Also, for $j=1,\dots,i$, we fix non-zero 
elements $w_{i,j+1}\in (I_j)_i$ in such a way that $w_{i,j}\mapsto w_{i+1,j}$ unless $j=i+1$ and
$w_{i,i+1}\mapsto 0$.

Let $A$ be the path algebra ${\bf C}Q$. Viewed as a representation of $Q$, $A$ is isomorphic to
the direct sum $\bigoplus_{i=1}^n P_i$.
In particular, ${\bf dim}\, A=(1,2,\dots,n)$.
The linear dual $A^*$ is isomorphic to the direct sum of injective representations $\bigoplus_{i=1}^n I_i$.

\begin{prop}\label{QGCF}
The quiver Grassmannian
${\rm Gr}_{{\bf dim}\,A}(A\oplus A^*)$ is isomorphic to the degenerate
flag variety $\Fl^a_{n+1}$ of $\mathfrak{sl}_{n+1}$.
\end{prop}
\begin{proof}
Consider $A\oplus A^*=\bigoplus_{i=1}^n (P_i\oplus I_i)$ as a representation of $Q$.
Let $W_j$ the the space attached to the $j$-th vertex, that is,
$A\oplus A^*=(W_1,\dots,W_n)$.  First, we note that $\dim W_j=n+1$ for all $j$.
Second, we fix an $(n+1)$-dimensional vector space $W$ with a basis $w_1,\dots,w_{n+1}$.
Let us identify all $W_j$ with $W$ by sending $w_{i,j}$ to $w_j$. Then the  maps $W_j\to W_{j+1}$
coincide with $pr_{j+1}$. Now our proposition follows from the equality ${\bf dim}\, A= (1,2,\dots,n)$.
\end{proof}

The coefficient quiver of the representation $A\oplus A^*$  is given by ($n=4$):
\begin{equation}\label{wij}
\xymatrix@R=6pt@C=8pt
{
w_{1,5}\ar[r]&w_{2,5}\ar[r]&w_{3,5}\ar[r]&w_{4,5}&\\
w_{1,4}\ar[r]&w_{2,4}\ar[r]&w_{3,4}&w_{4,4}&\\
w_{1,3}\ar[r]&w_{2,3}&w_{3,3}\ar[r]&w_{4,3}&\\
w_{1,2}&w_{2,2}\ar[r]&w_{3,2}\ar[r]&w_{4,2}&\\
w_{1,1}\ar[r]&w_{2,1}\ar[r]&w_{3,1}\ar[r]&w_{4,1}&\\
}
\end{equation}

\begin{rem}
We note that the classical $SL_{n+1}$ flag variety has a similar realization. Namely, let
$\tilde M$ be the representation of $Q$ isomorphic to the direct sum of $n+1$ copies of $P_1$
(so, ${\bf dim}\, \tilde M={\bf dim}\, (A\oplus A^*)$). Then the classical flag variety $SL_{n+1}/B$
is isomorphic to the quiver Grassmannian ${\rm Gr}_{{\bf dim}A} \tilde M$. The $Q$-representation $\tilde M$
can be visualized as ($n=4$)
\begin{equation}
\xymatrix@R=6pt@C=8pt
{
\bullet\ar[r]&\bullet\ar[r]&\bullet\ar[r]&\bullet&\\
\bullet\ar[r]&\bullet\ar[r]&\bullet\ar[r]&\bullet&\\
\bullet\ar[r]&\bullet\ar[r]&\bullet\ar[r]&\bullet&\\
\bullet\ar[r]&\bullet\ar[r]&\bullet\ar[r]&\bullet&\\
\bullet\ar[r]&\bullet\ar[r]&\bullet\ar[r]&\bullet&\\
}
\end{equation}
\end{rem}

We can easily generalize Proposition \ref{QGCF} to degenerate partial flag varieties:\\[1ex]
Suppose we are given a sequence ${\bf d}=(0=d_0<d_1<d_2<\ldots<d_s<d_{s+1}=n+1)$. Then we define
$$P=\bigoplus_{i=1}^s P_i^{d_i-d_{i-1}},\;\;\; I=\bigoplus_{i=1}^s I_i^{d_{i+1}-d_i}$$
as representations of an equioriented quiver of type $A_s$.

\begin{prop}\label{QGCFbd}
The quiver Grassmannian ${\rm Gr}_{{\bf dim}\,P}(P\oplus I)$ is isomorphic to the degenerate
partial flag variety $\Fl^a_\bd$ of $\mathfrak{sl}_{n+1}$.
\end{prop}
\begin{proof}
We note that the dimension vector of $P\oplus I$ is given by $(n+1,\dots,n+1)$ and the dimension
vector of $P$ equals $(d_1,\dots,d_s)$. Now let us identify the spaces $(P\oplus I)_j$ with
$W$ as in the proof of Proposition \ref{QGCF}. Then the map $(P\oplus I)_j\to (P\oplus I)_{j+1}$
corresponding to the arrow $j\to j+1$ coincides with $pr_{d_j+1}\dots pr_{d_{j+1}}$, which
proves the proposition.
\end{proof}

\section{A class of well-behaved quiver Grassmannians}\label{Sec:WellBehaved}

\subsection{Geometric properties} From now on, let $Q$ be a Dynkin quiver.
Then $G_{\bf d}$ acts with finitely many orbits on $R_{\bf d}(Q)$ for every ${\bf d}$;
in particular, for every ${\bf d}\in{\bf N}Q_0$, there exists a unique
(up to isomorphism) exceptional representation of this dimension vector.\\[1ex]
The subsets $\mathcal{S}_{[N]}$ of Section \ref{qcqg} then define a finite stratification of each quiver Grassmannian
${\rm Gr}_{\bf e}(M)$ according to isomorphism
type of the subrepresentation $N \subset M$.
\begin{prop}\label{ggp} 
Assume that $X$ and $Y$ are exceptional representations of $Q$ such that
${\rm Ext}_Q^1(X,Y)=0$. Define $M=X\oplus Y$ and ${\bf e}={\bf dim}\,X$, ${\bf d}={\bf dim}\,(X\oplus Y)$. 
Then the following holds:
\begin{enumerate}
\item $\dim{\rm Gr}_{\bf e}(M)=\langle {\bf e},{\bf d}-{\bf e}\rangle$.
\item The variety ${\rm Gr}_{\bf e}(M)$ is reduced, irreducible and rational.
\item ${\rm Gr}_{\bf e}(M)$ is a locally complete intersection scheme.
\end{enumerate}
\end{prop}
\begin{proof}
The representation $X$ obviously embeds into $M$, thus
$$\dim{\rm Gr}_{\bf e}(M)\geq\dim\mathcal{S}_{[X]}=
\dim{\rm Hom}_Q(X,M)-\dim{\rm End}_Q(X)=\dim{\rm Hom}_Q(X,Y).$$
The tangent space to any point $U\in\mathcal{S}_{[X]}$ has dimension
${\rm dim}\,{\rm Hom}_Q(X,Y)$, too, thus $\overline{\mathcal{S}_{[X]}}$ is reduced. Moreover, a generic embedding
of $X$ into $X\oplus Y$ is of the form $[{\rm id}_X,f]$ for a map $f\in{\rm Hom}_Q(X,Y)$,
and this identifies an open subset isomorphic to ${\rm Hom}_Q(X,Y)$ of $\mathcal{S}_{[X]}$,
proving rationality of $\overline{\mathcal{S}_{[X]}}$. 
Now suppose $N$ embeds into $M=X\oplus Y$ and ${\bf dim}\, N={\bf e}$. 
Then ${\rm Ext}^1_Q(N,Y)=0$
since ${\rm Ext}_Q^1(X\oplus Y,Y)=0$ by assumption, and thus $\dim{\rm Hom}_Q(N,Y)=
\langle {\bf e},{\bf d}-{\bf e}\rangle=\dim{\rm Hom}_Q(X,Y)$. Therefore, 
\begin{multline*}
\dim\mathcal{S}_{[N]}=\\
\dim{\rm Hom}_Q(N,X)-\dim{\rm Hom}_Q(N,N)+\dim{\rm Hom}_Q(X,Y)\leq\\ \dim{\rm Hom}_Q(X,Y),
\end{multline*}
which proves that $\dim{\rm Gr}_{\bf e}(M)=\dim{\rm Hom}_Q(X,Y)=\langle {\bf e},{\bf d}-{\bf e}\rangle$, and that
the closure of $\mathcal{S}_{[X]}$ is an irreducible component of ${\rm Gr}_{\bf e}(M)$. Conversely, suppose that an irreducible component $C$ of ${\rm Gr}_{\bf e}(M)$ is given, then necessarily $C$ is the closure of some stratum $\mathcal{S}_{[N]}$, and the dimension of $C$ equals $\langle {\bf e},{\bf d}-{\bf e}\rangle=\dim{\rm Hom}_Q(X,Y)$ by Proposition \ref{qgex}. By the above dimension estimate, we conclude $\dim{\rm Hom}_Q(N,X)=\dim{\rm Hom}_Q(N,N)$. By \cite[Theorem 2.4]{Bongartz}, this yields an embedding $N\subset X$, and thus $N=X$ by equality of dimensions. Therefore, ${\rm Gr}_{\bf e}(M)$ equals the closure of the stratum $\mathcal{S}_{[X]}$, thus it is irreducible, reduced and rational.
The dimension of ${\rm Hom}^0_Q({\bf e},M)$ equals
$\langle {\bf e},{\bf d}-{\bf  e}\rangle+\dim G_{\bf e}$, thus its codimension in
$R_{\bf e}(Q)\times{\rm Hom}^0_Q({\bf e},{\bf d})$ equals
$$\dim R_{\bf e}(Q)+\sum_ie_id_i-\langle {\bf e},{\bf d}-{\bf e}\rangle-
\dim G_{\bf e}=\sum_{(\alpha:i\rightarrow j)\in Q_1}e_id_j.$$
But this is exactly the number of equations defining ${\rm Hom}_Q({\bf e},M)$.
Thus ${\rm Hom}^0_Q({\bf e},M)$ is locally a complete intersection. 
The map ${\rm Hom}^0_Q({\bf e},M)\rightarrow {\rm Gr}_{\bf e}(M)$
is locally trivial with smooth fiber $G_{\bf e}$, hence the last statement follows.
\end{proof}

On a quiver Grassmannian ${\rm Gr}_{\bf e}(M)$, the automorphism group ${\rm Aut}_Q(M)$
acts algebraically. In the present situation, this implies that the group
$$G=\left[\begin{array}{cc}{\rm Aut}_Q(X)&0\\ {\rm Hom}_Q(X,Y)&{\rm Aut}_Q(Y)
\end{array}\right]$$
acts on ${\rm Gr}_{\bf e}(X\oplus Y)$.

\subsection{Flat degeneration}
Now let $\tilde{M}$ be the unique (up to isomorphism) exceptional representation of the same
dimension vector as $M$. By Proposition \ref{qgex}, we also have
$\dim {\rm Gr}_{\bf e}(\tilde{M})=\langle {\bf e},{\bf d}-{\bf e}\rangle$. It is thus reasonable
to ask for good properties of the degeneration from
${\rm Gr}_{\bf e}(\tilde{M})$ to ${\rm Gr}_{\bf e}(M)$.

\begin{thm}\label{flat} 
Under the previous hypotheses, the quiver Grassmannian ${\rm Gr}_{\bf e}(M)$ is a flat degeneration of ${\rm Gr}_{\bf e}(\tilde{M})$.
\end{thm}

\begin{proof} Let $Y$ be the open subset of $R_{\bf d}(Q)$ consisting of all representations $Z$
whose orbit closure $\overline{\mathcal{O}_Z}$ contains the orbit $\mathcal{O}_M$; in particular,
$Y$ contains $\mathcal{O}_{\tilde{M}}$.
We consider the diagram
$${\rm Gr}_{\bf e}({\bf d})\stackrel{p_1}{\leftarrow}{\rm Gr}_{\bf e}^Q({\bf d})\stackrel{p_2}{\rightarrow}R_{\bf d}(Q)$$
of the previous section. In particular, we consider the restriction $q:\tilde{Y}\rightarrow Y$ of $p_2$
to $\tilde{Y}=p_2^{-1}(Y)$. This is a proper morphism (since $p_2$ is so) between two smooth and
irreducible varieties (since they are open subsets of the smooth varieties $R_{\bf d}(Q)$ and
${\rm Gr}_{\bf e}^Q({\bf d})$, respectively). The general fibre of
$q$ is ${\rm Gr}_{\bf e}(\tilde{M})$, since the orbit of $\tilde{M}$, being exceptional,
is open in $Y$, and the special fibre of $q$ is ${\rm Gr}_{\bf e}(M)$, since the orbit of $M$
is closed in $Y$ by definition. By semicontinuity, all fibres of $q$ have the same dimension
$\langle {\bf e},{\bf d}-{\bf e}\rangle$. By \cite[Corollary to Theorem 23.1]{Mats},
a proper morphism between smooth and irreducible varieties with constant fibre dimension is
already flat.
\end{proof}

\begin{rem}
Theorem \ref{flat} generalizes Proposition 3.15 of \cite{F1} (see also subsections \ref{ss1}, \ref{comp}),
where the flatness of the degeneration $\Fl_n\to \Fl^a_n$ was proved using complicated combinatorial tools.
\end{rem}

Note that the degeneration from $\tilde{M}$ to $M$ in $R_{\bf d}(Q)$
can be realized along a one-parameter subgroup of $G_{\bf d}$ in the following way:

\begin{lem} Under the above hypothesis, there exists a short exact sequence
$0\rightarrow X\rightarrow \tilde{M}\rightarrow Y\rightarrow 0$.
\end{lem}
\begin{proof}
By \cite[Theorem 3.3]{SchofieldGeneric}, a generic representation $Z$ of dimension vector
${\bf d}$ admits a subrepresentation of dimension vector ${\bf e}$ if ${\rm Ext}_Q^1(N,Q)$
vanishes for generic $N$ of dimension vector ${\bf e}$ and generic $Q$ of dimension vector
${\bf d}-{\bf e}$. In the present case, these generic representations are $Z=\tilde{M}$, $N=X$
and $Q=Y$, and the lemma follows.
\end{proof}

This lemma implies that $\tilde{M}$ can be written, up to isomorphism,
in the following form
$$\tilde{M}_\alpha=\left[\begin{array}{cc}X_\alpha&\zeta_\alpha\\ 0&Y_\alpha\end{array}\right]$$
for all $\alpha\in Q_1$. Conjugating with the one-parameter subgroup
$$\left(\left[\begin{array}{cc}t\cdot{\rm id}X_i&0\\ 0&{\rm id}Y_i\end{array}\right]\right)_{i\in Q_0}$$
of $G_{\bf d}$ and passing to the limit $t=0$, we arrive at the desired degeneration.\\[2ex]
Since $Q$ is a Dynkin quiver, the isomorphism classes of indecomposable representations of $Q$
are parametrized by the positive roots $\Phi^+$ of the corresponding root system.
We view $\Phi^+$ as a subset of ${\bf N}Q_0$ by identifying the simple root $\alpha_i$ with 
with the vector having $1$ at the $i$-th place and zeros everywhere else.
Denote by $V_\alpha$ the indecomposable representation corresponding to $\alpha\in\Phi^+$;
more precisely, ${\bf dim} V_\alpha=\alpha$. Using this parametrization of the
indecomposables and the Auslander-Reiten quiver of $Q$, we can actually construct
$\tilde{M}$ explicitly from $X$ and $Y$ (or, more precisely, from their decompositions
into indecomposables), using the algorithm of
\cite[Section 3]{ReGenExt}.


\section{The group action and normality}\label{CGA}
In this section we put $X=P$ and $Y=I$, where $P$ and $I$ are projective and injective representations of a Dynkin quiver  $Q$. We consider the group
$$G=\left[\begin{array}{cc}{\rm Aut}_Q(P)&0\\ {\rm Hom}_Q(P,I)&{\rm Aut}_Q(I)\end{array}\right].$$

\begin{thm}\label{Thm:GroupAction} The group $G$ acts on ${\rm Gr}_{{\bf dim} P}(P\oplus I)$ with finitely many orbits, parametrized by pairs of isomorphism classes $([Q_P],[N_I])$ such that $Q_P$ is a quotient of $P$, $N_I$ is a subrepresentation of $I$, and $Q_P$ and $N_I$ have the same dimension vector.
\end{thm}

\begin{proof} Suppose $N$ is a subrepresentation of $P\oplus I$ of dimension vector ${\bf dim}\,N={\bf dim }\,P$, and denote by $\iota:N\rightarrow P\oplus I$ the embedding. Define $N_I=N\cap I$ and $N_P=N/(N\cap I)$. Then $N_P\simeq(N+I)/I$ embeds into $(P\oplus I)/I\simeq P$, thus $N_P$ is projective since ${\rm rep}(Q)$ is hereditary. Therefore, the short exact sequence
$$0\rightarrow N_I\rightarrow N\rightarrow N_P\rightarrow 0$$
splits. We thus have a retraction $r:N_P\rightarrow N$ such that $N$ is the direct sum of $N_I$ and $r(N_P)$, and such that $N_I$ embeds into the component $I$ of $P\oplus I$ under $\iota$. Without loss of generality, we can thus write the embedding of $N$ into $P\oplus I$ as
$$\iota=\left[\begin{array}{cc}\iota_P&0\\ f&\iota_I\end{array}\right]:N_P\oplus N_I\rightarrow P\oplus I$$
for $\iota_P$ (resp. $\iota_I$) an embedding of $N_P$ (resp. $N_I$) into $P$ (resp. $I$), and $f:N_P\rightarrow I$. Since $I$ is injective, the map $f$ factors through $\iota_P$, yielding a map $x:P\rightarrow I$ such that $x\iota_P=f$. We can then conjugate $\iota$ with
$$\left[\begin{array}{cc}1&0\\ -x&1\end{array}\right]\in G.$$
We have thus proved that each $G$-orbit in ${\rm Gr}_{{\bf dim}_P}(P\oplus I)$ contains an embedding of the form
$$\left[\begin{array}{cc}\iota_P&0\\ 0&\iota_I\end{array}\right]:N_P\oplus N_I\rightarrow P\oplus I,$$
such that $N_I$ is a subrepresentation of $I$, the representation $Q_P=P/N_P$ is a quotient of $P$, and their dimension vectors obviously add up to ${\bf dim}\,P$. We now have to show that the isomorphism classes of such $Q_P$ and $N_I$ already characterize the corresponding $G$-orbit in ${\rm Gr}_{{\bf dim}\,P}(P\oplus I)$. To do this, suppose we are given two such embeddings
$$\left[\begin{array}{cc}\iota_P&0\\ 0&\iota_I\end{array}\right]:N_P\oplus N_I\rightarrow P\oplus I,\;\;\;
\left[\begin{array}{cc}\iota_P'&0\\ 0&\iota_I'\end{array}\right]:N_P'\oplus N_I'\rightarrow P\oplus I$$
such that the cokernels $Q_P$ and $Q_P'$ of $\iota_P$ and $\iota_P'$, respectively, are isomorphic, and such that $N_I$ and $N_I'$ are isomorphic. By \cite[Lemma 6.3]{ReFM}, an arbitrary isomorphism $\psi_I:N_I\rightarrow N_I'$ lifts to an automorphism $\varphi_I$ of $I$, such that $\varphi_I\iota_I=\iota_I'\psi_I$. By the obvious dual version of the same lemma, an arbitrary isomorphism $\xi_P:Q_P\rightarrow Q_P'$ lifts to an automorphism $\varphi_P$ of $P$, which in turn induces an isomorphism $\psi_P:N_P\rightarrow N_P'$ such that $\varphi_P\iota_P=\iota_P'\psi_P$. This proves that the two embeddings above are conjugate under $G$. Finally, given representations $Q_P$ and $N_I$ as above, we can define $N_P$ as the kernel of the quotient map and get an embedding as above.
\end{proof}

\begin{rem}\label{two}
We can obtain an explicit parametrization of the orbits as follows: we write
$$P=\bigoplus_{i\in Q_0}P_i^{a_i},\;\;\; I=\bigoplus_{i\in Q_0}I_i^{b_i}.$$
By \cite[Lemma 4.1]{ReFM} and its obvious dual version, we have:\\[1ex]
A representation $N_I$ embeds into $I$ if and only if $\dim{\rm Hom}_Q(S_i,N_I)\leq b_i$ for all $i\in Q_0$,\\[1ex]
a representation $Q_P$ is a quotient of $P$ if and only if $\dim{\rm Hom}_Q(Q_P,S_i)\leq a_i$ for all $i\in Q_0$.
\end{rem}

The previous result establishes a finite decomposition of the quiver Grassmannians into orbits. In particular the tangent space is equidimensional along every such orbit. The following examples shows that in general such orbits are not cells.
\begin{example}\label{D4}
Let
$$
\xymatrix@R=3pt{
   & & &3&\\
Q:=&1\ar[r]&2\ar[rr]\ar[ru]& &4
}
$$
be a Dynkin quiver of type $D_4$. The quiver Grassmannian ${\rm Gr}_{(1211)}(I_3\oplus I_4)$ is isomorphic to ${\bf P}^1$, with the points $0$ and $\infty$ corresponding to two decomposable representations, whereas all points in ${\bf P}^1\setminus\{0,\infty\}$, which is obviously not a cell, correspond to subrepresentations which are isomorphic to the indecomposable representation of dimension vector $(1211)$.\\[1ex]
\end{example}

We note the following generalization of the tautological bundles
$$\mathcal{U}_i=\{(U,x)\in{\rm Gr}_{\bf e}(X)\times X_i\,:\, x\in U_i\}$$
on ${\rm Gr}_{\bf e}(X)$:\\[1ex]
Given a projective representation $P$, the trivial vector bundle ${\rm Hom}_Q(P,X)$ on ${\rm Gr}_{\bf e}(X)$ admits the subbundle $$\mathcal{V}_P=\{(U,\alpha)\in{\rm Gr}_{\bf e}(X)\times{\rm Hom}_Q(P,X)\,:\; \alpha(P)\subset U\}.$$
We then have $\mathcal{V}_P\simeq \bigoplus_{i\in Q_0}\mathcal{V}_i^{m_i}$ if $P\simeq\bigoplus_{i\in Q_0}P_i^{m_i}$. Dually, given an injective representation $I$, the trivial vector bundle ${\rm Hom}_Q(X,I)$ admits the subbundle $$\mathcal{V}_I=\{(U,\beta)\in{\rm Gr}_e(X)\times{\rm Hom}_Q(X,I)\,:\; \beta(U)=0\}.$$
We then have $\mathcal{V}_I\simeq \bigoplus_{i\in Q_0}(\mathcal{V}_i^*)^{m_i}$ if $I\simeq\bigoplus_{i\in Q_0}I_i^{m_i}$.\\[1ex]
Given a decomposition of the dimension vector ${\bf dim}\, P=\be=\bff + \bg$, recall
the subvariety ${\mathcal S}_{\bf f}(P\oplus I)\subset {\rm Gr}_{\be}(P\oplus I)$ consisting of all representations $N$ such that
${\bf dim}\, N\cap I=\bff$ and ${\bf dim}\, \pi(N)=\bg$, where $\pi:P\oplus I\to P$ is the natural projection.
We have a natural surjective map $\zeta:{\rm Gr}_{\bff,\bg}(P\oplus I)\to {\rm Gr}_{\bg}(P)\times {\rm Gr}_{\bff}(I)$.
We note that since $P$ is projective, all the points of
${\rm Gr}_{\bg}(P)$ are isomorphic as representations of $Q$. Also, since $I$ is injective, for any two points
$M_1,M_2\in {\rm Gr}_{\bff}(I)$ the representations $I/M_1$ and $I/M_2$ of $Q$ are isomorphic.
Therefore, the dimension of the vector space $\Hom_Q(N_P,I/N_I)$ is independent of
the points $N_P\in {\rm Gr}_{\bg}(P)$ and $N_I\in {\rm Gr}_{\bff}(I)$. We denote this dimension by $D$.
\begin{prop}\label{Prop:ZetaBundle}
The map $\zeta$ is a $D$-dimensional vector bundle (in the Zariski topology).
\end{prop}
\begin{proof}
Associated to $N_P$ and $N_I$, we have exact sequences
$$0\rightarrow N_P\rightarrow P\rightarrow Q_P\rightarrow 0\quad \mbox{ and } 
\quad 0\rightarrow N_I\rightarrow I\rightarrow Q_I\rightarrow 0.$$
These induce the following commutative diagram with exact rows and columns (the final zeroes arising from projectivity of $N_P$ and injectivity of $Q_I$; we abbreviate ${\rm Hom}_Q(\_,\_)$ by $(\_,\_)$):
$$\begin{array}{ccccccccc}
&&0&&0&&0&&\\
&&\downarrow&&\downarrow&&\downarrow&&\\
0&\rightarrow&(Q_P,N_I)&\rightarrow&(Q_P,I)&\rightarrow&(Q_P,Q_I)&&\\
&&\downarrow&&\downarrow&&\downarrow&&\\
0&\rightarrow&(P,N_I)&\rightarrow&(P,I)&\rightarrow&(P,Q_I)&\rightarrow&0\\
&&\downarrow&&\downarrow&&\downarrow&&\\
0&\rightarrow&(N_P,N_I)&\rightarrow&(N_P,I)&\rightarrow&(N_P,Q_I)&\rightarrow&0\\
&&&&\downarrow&&\downarrow&&\\
&&&&0&&0&&\end{array}$$
This diagram yields an isomorphism
$${\rm Hom}_Q(N_P,Q_I)\simeq{\rm Hom}_Q(P,I)/({\rm Hom}_Q(P,N_I)+{\rm Hom}_Q(Q_P,I)).$$
Pulling back the tautological bundles constructed above via the projections
$${\rm Gr}_{\bf g}(P)\stackrel{{\rm pr}_1}{\leftarrow}{\rm Gr}_{\bf g}(P)\times{\rm Gr}_{\bf f}(I)\stackrel{{\rm pr}_2}{\rightarrow}{\rm Gr}_{\bf f}(I),$$
we get subbundles ${\rm pr}_2^*\mathcal{V}_P$ and ${\rm pr}_1^*\mathcal{V}_I$ of the trivial bundle ${\rm Hom}_Q(P,I)$ on
${\rm Gr}_{\bf g}(P)\times{\rm Gr}_{\bf f}(I)$. By the above isomorphism, the quotient bundle
$${\rm Hom}_Q(P,I)/({\rm pr}_1^*\mathcal{V}_P+{\rm pr}_2^*\mathcal{V}_I)$$
identifies with the fibration $\zeta:{\mathcal S}_{\bf f}(P\oplus I)\to {\rm Gr}_{\bg}(P)\times {\rm Gr}_{\bff}(I)$, proving Zariski local triviality of the latter.
\end{proof}

The methods established in the two previous proofs now allow us to prove normality of the quiver Grassmannians.

\begin{thm} The quiver Grassmannian ${\rm Gr}_{\bf e}(P\oplus I)$ is a normal variety.
\end{thm}

\begin{proof}
We already know that ${\rm Gr}_{\bf e}(P\oplus I)$ is locally a complete intersection,
thus normality is proved once we know that ${\rm Gr}_{\bf e}(P\oplus I)$ is regular in codimension $1$.
By the proof of Theorem \ref{Thm:GroupAction}, we know that a subrepresentation $N$ of $P\oplus I$ of dimension vector ${\bf dim} P$
is of the form $N=N_P\oplus N_I$,
with exact sequences
$$0\rightarrow N_P\rightarrow P\rightarrow Q_P\rightarrow 0,\;\;\;
0\rightarrow N_I\rightarrow I\rightarrow Q_I\rightarrow 0,$$
such that $N_I$ and $Q_P$ are of the same dimension vector ${\bf f}$.
By the tangent space formula, $N$ defines a singular point of ${\rm Gr}_{\bf e}(P\oplus I)$ if and
only if 
$${\rm Ext}_Q^1(N_P\oplus N_I,Q_P\oplus Q_I)={\rm Ext}^1_Q(N_I,Q_P)$$ is non-zero.
In particular, singularity of the point $N$ only depends on the isomorphism types of $N_I=N\cap I$ and
$Q_P=(P\oplus I)/(N+I)$. Consider the locally closed subset $Z$ of ${\rm Gr}_{\bf e}(P\oplus I)$ consisting of subrepresentations $N'$
such that $N'\cap I\simeq N_I$ and $(P\oplus I)/(N'+I)\simeq Q_P$;
thus $Z\subset\mathcal{S}_{\bf f}$. The vector bundle
$\zeta:\mathcal{S}_{\bf f}\rightarrow {\rm Gr}_{\bf f}(I)\times{\rm Gr}_{{\bf e}-{\bf f}}(P)$
of the previous proposition restricts to a vector bundle $\zeta:Z\rightarrow Z_I\times Z_P$, where $Z_I=\mathcal{S}_{[N_I]}\subset{\rm Gr}_{\bf f}(I)$ consists of subrepresentations isomorphic to
$N_I$, and $Z_P\subset{\rm Gr}_{{\bf e}-{\bf f}}(P)$ consists of subrepresentations with quotient
isomorphic to $Q_P$. By the dimension formula for the strata $\mathcal{S}_{[N_I]}$, the codimension
of $Z_I$ in ${\rm Gr}_{\bf f}(I)$ equals $\dim{\rm Ext}_Q^1(N_I,N_I)$; dually, the codimension of
$Z_P$ in ${\rm Gr}_{{\bf e}-{\bf f}}(P)$ equals $\dim{\rm Ext}_Q^1(Q_P,Q_P)$. Since the rank of
the bundle $\zeta$ is $\dim{\rm Hom}_Q(N_P,Q_I)$, we have
$$\dim{\rm Gr}_{\bf e}(P\oplus I)-\dim\zeta^{-1}(Z_I\times Z_P)=$$
$$=\dim{\rm Gr}_{\bf e}(P\oplus I)-\dim{\rm Hom}_Q(N_P,Q_I)-(\dim{\rm Gr}_{\bf f}(I)-\dim{\rm Ext}_Q^1(N_I,N_I))-$$
$$-(\dim{\rm Gr}_{{\bf e}-{\bf f}}(P)-\dim{\rm Ext}_Q^1(Q_P,Q_P))=$$
$$=\langle {\bf e},{\bf d}\rangle-\langle {\bf e}-{\bf f},{\bf d}-{\bf f}\rangle-\langle{\bf f},{\bf d}-{\bf f}\rangle-\langle{\bf e}-{\bf f},{\bf f}\rangle+\dim{\rm Ext}_Q^1(N_I,N_I)+\dim{\rm Ext}_Q^1(Q_P,Q_P)=$$
$$=\langle {\bf f},{\bf f}\rangle+\dim{\rm Ext}_Q^1(N_I,N_I)+\dim{\rm Ext}_Q^1(Q_P,Q_P)$$
for the codimension of $Z$ in ${\rm Gr}_{\bf e}(P\oplus I)$.
Assume that this codimension equals $1$. Since the Euler form ($Q$ being Dynkin) is positive definite,
the summand $\langle {\bf f},{\bf f}\rangle$ is nonnegative. If it equals $0$, then ${\bf f}$ equals
$0$, and $N_I$ and $Q_P$ are just the zero representations, a contradiction to the assumption ${\rm Ext}_Q^1(N_I,Q_P)\not=0$.
Thus $\langle{\bf f},{\bf f}\rangle=1$ and both other summands are zero, thus $N_I$ and $Q_P$ are both
isomorphic to the exceptional representation of dimension vector ${\bf f}$. But this implies
vanishing of ${\rm Ext}^1_Q(N_I,Q_P)$ and thus nonsingularity of $N$.
\end{proof}

\section{Cell decomposition}\label{1dtorus}

Let $Q$ be a Dynkin quiver, $P$ and $I$ respectively a projective and an injective representation of $Q$.
Let $M:=P\oplus I$ and let $X={\rm Gr}_{\mathbf{e}}(M)$ where $\mathbf{e}=\mathbf{dim}\, P$. In this section
we construct a cellular decomposition of $X$.

The indecomposable direct summands of $M$ are either injective or projective. In particular
they are thin, that is, the vector space at every vertex is at most one--dimensional. The set of generators
of these one--dimensional spaces form a linear basis of $M$ which we denote by $\mathbf{B}$.
To each indecomposable summand $L$ of $M$ we assign an integer $d(L)$, the \emph{degree} of L,
so that if ${\rm Hom}_Q(L,L')\neq 0$ then $d(L)<d(L')$ and so that all the degrees are different. In particular the degrees of the homogeneous vectors of $I$ are strictly bigger than that ones of $P$ (in case there is a projective--injective summand in both $P$ and $I$ we chose the degree of the copy in $I$ to be bigger than the degree of the copy in $P$). To every vector of $L$ we assign degree $d(L)$. In particular every element $v$ of $\mathbf{B}$ has an assigned degree $d(v)$. In view of \cite{string} the one--dimensional torus $T=\mathbf{C}^\ast$ acts on $X$ as follows: for every
$v\in \mathbf{B}$ and every $\lambda\in T$ we define
\begin{equation}\label{Eq:TorusAction}
\lambda\cdot v:=\lambda^{d(v)}v.
\end{equation}
This action extends uniquely to an action on $M$ and induces an action on $X$. The $T$--fixed points
are precisely the points of $X$ generated by a part of $\mathbf{B}$, that is, the coordinate sub--representations of $P\oplus I$ of dimension vector ${\bf dim}\, P$.

We denote the (finite) set of $T$--fixed points of $X$ by $X^T$.

For every $L\in X^T$
the torus acts on the tangent space $T_L(X)\simeq{\rm Hom}_Q(L,M/L)$. More explicitly, the vector
space $\textrm{Hom(L,M/L)}$ has a basis given by elements which associate to a basis vector
$v\in L\cap \mathbf{B}$ a non--zero element $v'\in M/L\cap\mathbf{B}$ and such element is homogeneous
of degree
$d(v')-d(v)$ \cite{CrawleyTree}. We denote by ${\rm Hom}_Q(L,M/L)^+$ the vector subspace of ${\rm Hom}_Q(L,M/L)$ generated by the
basis elements of \emph{positive} degree.

Since $X$ is projective, for every $N\in X$ the limit $\lim_{\lambda\rightarrow 0}\lambda.N$ exists and moreover it is $T$--fixed (see e.g. \cite[Lemma~2.4.3]{Chriss}).
For every $L\in X^T$ we consider its \emph{attracting set}
$$
\mathcal{C}(L)=\{N\in X|\,\lim_{\lambda\rightarrow0}\lambda\cdot N=L\}.
$$
The action \eqref{Eq:TorusAction} on $X$ induces an action on ${\rm Gr}_\mathbf{f}(I)$
and ${\rm Gr}_{\mathbf{e}-\mathbf{f}}(P)$ so that the map
\begin{equation}\label{Eq:ZetaCell}
\zeta:{\mathcal S}_{\mathbf{f}}\rightarrow {\rm Gr}_\mathbf{f}(I)\times {\rm Gr}_{\mathbf{e}-\mathbf{f}}(P)
\end{equation}
is $T$--equivariant. Since both ${\rm Gr}_\mathbf{f}(I)$
and ${\rm Gr}_{\mathbf{e}-\mathbf{f}}(P)$ are smooth ($P$ and $I$ being rigid), we apply  \cite{BB} and we get cellular decompositions into attracting sets
$$
\begin{array}{ccc}
{\rm Gr}_\mathbf{f}(I)=\coprod_{L_I\in {\rm Gr}_\mathbf{f}(I)^T}\mathcal{C}(L_I)
&\text{and}
&{\rm Gr}_\mathbf{g}(P)=\coprod_{L_P\in {\rm Gr}_\mathbf{g}(P)^T}\mathcal{C}(L_P)
\end{array}
$$
and moreover $\mathcal{C}(L_I)\simeq {\rm Hom}_Q(L_I,I/L_I)^+$ and $\mathcal{C}(L_P)\simeq {\rm Hom}_Q(L_P,P/L_P)^+$.

\begin{thm}\label{Thm:CellDec}
For every $L\in X^T$ its attracting set is an affine space isomorphic to
${\rm Hom}_Q(L,M/L)^+$. In particular we get a cellular decomposition
$X=\coprod_{L\in X^T}\mathcal{C}(L)$. Moreover
\begin{equation}\label{Eq:CellDescr}
\mathcal{C}(L)=\zeta^{-1}(\mathcal{C}(L_I)\times\mathcal{C}(L_P))\simeq\mathcal{C}(L_I)\times \mathcal{C}(L_P)\times {\rm Hom}_Q(L_P,I/L_I).
\end{equation}
\end{thm}
\begin{proof}
The subvariety ${\mathcal S}_{\mathbf{f}}:=\zeta^{-1}({\rm Gr}_\mathbf{f}(I)\times {\rm Gr}_{{\bf dim}\, P-\mathbf{f}}(P))$ is smooth but not projective. Nevertheless it enjoys the following property
\begin{equation}\label{Eq:ClosedLimGrfg}
\textrm{for all }\; N\in {\mathcal S}_{\mathbf{f}},\;\;\lim_{\lambda\rightarrow0}\lambda\cdot 
N\in {\mathcal S}_{\mathbf{f}}.
\end{equation}
Indeed let $N$ be a point of ${\mathcal S}_{\mathbf{f}}$ and let $w_1,\cdots, w_{|\mathbf{e}|}$ be a basis of it (here $|\mathbf{e}|=\sum_{i\in Q_0}e_i$). We write every $w_i$ in the basis $\mathbf{B}$ and we find a vector $v_i\in \mathbf{B}$ which has minimal degree  in this linear combination and whose coefficient can be assumed to be $1$. We call $v_i$ the leading term of $w_i$. The sub--representation $N_I=N\cap I$ is generated by those $w_i$'s which belong to $I$ while $N_P=\pi(N)\simeq N/N_I$ is generated by the remaining ones. The torus action is chosen in such a way that the leading term of every $w_j\in N_P$ belongs to P. The limit point $L:=\lim_{\lambda\rightarrow0}\lambda\cdot N$ has $v_1,\cdots, v_{|\mathbf{e}|}$ as its basis. The sub--representation $L_I=L\cap I$ is generated precisely by the $v_i$'s which are the leading terms of $w_i\in N_I$. In particular $\mathbf{dim}\,L_I=\mathbf{dim}\,N_I=\mathbf{f}$ and hence $L\in X_\mathbf{f}$.

Since the map $\zeta$ is $T$--equivariant, \eqref{Eq:CellDescr} follows from \eqref{Eq:ClosedLimGrfg}.

It remains to prove that $\mathcal{C}(L)\simeq{\rm Hom}_Q(L,M/L)^+$. This is a consequence of the following
$$
\begin{array}{cc}
\mathcal{C}(L_I)\simeq {\rm Hom}_Q(L_I,I/L_I)^+,&
\mathcal{C}(L_P)\simeq {\rm Hom}_Q(L_P,P/L_P)^+\\&\\
{\rm Hom}_Q(L_P,I/L_I)^+={\rm Hom}_Q(L_P,I/L_I), &
{\rm Hom}_Q(L_I,P/L_I)^+=0,
\end{array}
$$
together with the isomorphism \eqref{Eq:CellDescr}.
\end{proof}

The following example shows that for $L\in X^T$ and $N\in \mathcal{C}(L)$ it is not true that the tangent spaces at N and L have the same dimension.
\begin{example}\label{Example:CellsTangent}
Let
$$
\xymatrix@R=3pt{
   & & &3&\\
Q:=&1\ar[r]&2\ar[rr]\ar[ru]& &4
}
$$
be a Dynkin quiver of type $D_4$. For every vertex $k\in Q_0$ let $P_k$ and $I_k$
be respectively the indecomposable projective and injective Q--representation at vertex $k$. Let $P:=P_1\oplus P_2\oplus P_3\oplus P_4$ and $I:=I_1\oplus I_2\oplus I_3\oplus I_4$. We consider the variety ${\rm Gr}_{(1233)}(I\oplus P)$. We assign degree $deg(P_k):=4-k$ and $deg(I_k):=4+k$ for $k=1,2,3,4$.
We notice that $I_4\oplus I_3\oplus I_2$ has an indecomposable sub--representation  $N_I$ of dimension vector $(1211)$ such that $\lim_{\lambda\rightarrow0}\lambda\cdot N_I=I_4\oplus (0110):=L$, where $(0110)$ denotes the indecomposable sub--representation of $I_3$ of dimension vector $(0110)$. We have $I/L_I\simeq I/N_I\simeq I_1\oplus I_1\oplus I_2$ and $\mathrm{dim}\,{\rm Hom}_Q(N_I,I/N_I)=\mathrm{dim}\,{\rm Hom}_Q(L_I,I/L_I)=3$. Let us choose $L_P$ inside $P$ of dimension vector $(0022)$ so that $L_I\oplus L_P\in X$.
We choose $L_P\simeq P_3^2\oplus P_4^2$ where $P_3^2$ is a sub--representation of $P_1\oplus P_3$ and $P_4^2$ is in $P_1\oplus P_2$. The quotient $P/L_P\simeq I_2\oplus (0110)\oplus P_4$. Now $\textrm{dim}\,{\rm Ext}^1_Q(N_I,P/L_P)=\textrm{dim}\,{\rm Ext}^1_Q(N_I,P_4)=1$ but
$\textrm{dim}\,{\rm Ext}^1_Q(L_I,P/L_P)=\textrm{dim}\,{\rm Ext}^1_Q(I_4,(0110))+
\textrm{dim}\,{\rm Ext}^1_Q((0110),P_4)=2$.
\end{example}

\section{Poincar\'e polynomials in type $A$ and Genocchi numbers}\label{PPEC}
In this section we compute the Poincar\'e polynomials of ${\rm Gr}_{{\bf dim}\, P} (P\oplus I)$ 
for equioriented quiver of type $A$ and derive some combinatorial consequences.

\subsection{Equioriented quiver of type $A$}
For two non-negative integers $N$ and $M$ the $q$-binomial coefficient $\binom{N}{M}_q$ is defined
by the formula
\[
\binom{N}{M}_q=\frac{N_q!}{M_q!(N-M)_q!},\mbox{ where } k_q!=(1-q)(1-q^2)\dots (1-q^k).
\]
We also set $\binom{N}{M}_q=0$ if $N<M$ or $N<0$ or $M<0$.

Recall (see Proposition \ref{QGCF}) that $\Fl^a_{n+1}$ is isomorphic to ${\rm Gr}_{{\bf dim}\, P} (P\oplus I)$,
where $P$ (resp. $I$) is the direct sum of all projective (resp. injective) indecomposable
representations of $Q$. According to Proposition \ref{Prop:ZetaBundle}, in order to compute the Poincar\'e polynomial of ${\rm Gr}_{\bf e}(P\oplus I)$, we only need to compute the Poincar\'e polynomials of
${\rm Gr}_{\bg} (P)$ and ${\rm Gr}_{\bff} (I)$ for arbitrary dimension vectors $\bg=(g_1,\dots,g_n)$
and $\bff=(f_1,\dots,f_n)$. Let us compute these polynomials in a slightly more general settings.
Namely, fix two collections of non-negative integers $a_1,\dots,a_n$ and $b_1,\dots,b_n$
and set $P=\bigoplus_{i=1}^n P_i^{a_i}$, $I=\bigoplus_{i=1}^n I_i^{b_i}$.

\begin{lem}
\begin{gather}
\label{P}
P_{\textrm{Gr}_\mathbf{g}(P)}(q)=
\prod_{k=1}^{n} \binom{a_1+\dots+a_k-g_{k-1}}{g_k-g_{k-1}}_q,\\
\label{I}
P_{\textrm{Gr}_\mathbf{f}(I)}(q)=
\prod_{k=1}^{n} \binom{b_{n+1-k}+f_{n+2-k}}{f_{n+1-k}}_q
\end{gather}
with the convention $g_0=0$, $f_{n+1}=0$.
\end{lem}
\begin{proof}
We first prove the first formula by induction on $n$. For $n=1$ the formula reduces to the well-known formula
for the  Poincar\'e polynomials of the classical Grassmannians. Let $n>1$.
Consider the map ${\rm Gr}_\bg(P)\to {\rm Gr}_{g_1}((P\oplus I)_1)$. We claim that this map is a fibration
with the base ${\rm Gr}_{g_1}(\bC^{a_1})$ and a fiber isomorphic to
${\rm Gr}_{(g_2-g_1,g_3-g_1,\dots,g_n-g_1)}(P_1^{a_1-g_1}\bigoplus_{i=2}^n P_i^{a_i})$.
In fact, an element of ${\rm Gr}_\bg(P)$ is a collection of spaces $(V_1,\dots,V_n)$ such that
$V_i\subset (P)_i$.
We note that all the maps in $P$ corresponding to the arrows $i\to i+1$ are embeddings.
Therefore, if one fixes a $g_1$-dimensional subspace  $V_1\subset P_{1}$, this automatically determines the $g_1$-dimensional subspaces to be contained in $V_2,\dots,V_n$. This proves the claim. Now  formula
\eqref{P} follows by induction.

In order to prove \eqref{I},  we consider the map
$$
\textrm{Gr}_{\mathbf{f}}(I)\rightarrow \textrm{Gr}_{\mathbf{f}^\ast}(I^\ast): N\mapsto\{\varphi\in I^\ast|\, \varphi(N)=0\},
$$
where $I^\ast=\textrm{Hom}_\mathbf{C}(I,\mathbf{C})$ and
$\mathbf{f}^\ast=(f_1^\ast,\cdots, f_{n-1}^\ast)=\textbf{dim}\, I-\mathbf{f}$ is defined by
$$
f_k^\ast=b_k+b_{k+1}+\dots + b_n-f_k.
$$
Now $I^\ast$ can be identified with $\bigoplus_{i=1}^n P_{n+1-i}^{b_i}$ by acting on the vertices of
$Q$ with the permutation $\omega:i\mapsto n-i$ for every $i=1,2,\cdots, n-1$. We hence have an isomorphism
$$
\textrm{Gr}_{\mathbf{f}}(I)\simeq \textrm{Gr}_{\omega \mathbf{f}^\ast}(\bigoplus_{i=1}^n P_i^{b_i}).
$$
Substituting in \eqref{P}, we obtain \eqref{I}.
\end{proof}

\begin{thm}
Let $X={\rm Gr}_\mathbf{e}(I\oplus P)$ with $I$ and $P$ as above. Then the Poincar\'e polynomial of $X$ is
given by $P_X(q)=$
\begin{equation}\label{PP}
=\sum_{\mathbf{f}+\mathbf{g}=\mathbf{e}}q^{\sum_{i=1}^n g_i(a_i-f_i+f_{i+1})}
\prod_{k=1}^{n} \binom{a_1+\dots+a_k-g_{k-1}}{g_k-g_{k-1}}_q
\prod_{k=1}^{n} \binom{b_{n+1-k}+f_{n+2-k}}{f_{n+1-k}}_q.
\end{equation}
\end{thm}
\begin{proof}
Recall the decomposition ${\rm Gr}_\mathbf{e}(P\oplus I)=\sqcup_{\bff} {\mathcal S}_{\bff}$.
Each stratum ${\mathcal S}_{\bff}$ is a total space of a vector bundle over
${\rm Gr}_\bg (P)\times {\rm Gr}_\bff(I)$ with fiber over a point
$(N_P,N_I)\in {\rm Gr}_\bg (P)\times {\rm Gr}_\bff(I)$
isomorphic to $\Hom_Q(N_P,I/N_I)$. Since $\mathrm{Ext}^1_Q(N_P,I/N_I)=0$, we obtain
$\dim \Hom_Q(N_P,I/N_I)=\langle\textbf{g},\textbf{dim}\,I-\mathbf{f}\rangle$.
Since $Q$ is the equioriented quiver of type $A_n$, we obtain
\[
\langle\textbf{g},\textbf{dim}\,I-\mathbf{f}\rangle=\sum_{i=1}^n g_i(a_i-f_i+f_{i+1}).
\]
Now our Theorem follows from formulas \eqref{P} and \eqref{I}.
\end{proof}

Now let $a_i=b_i=1$, $i=1,\dots,n$. Then the quiver Grassmannian ${\rm Gr}_{{\bf dim}\, P}(P\oplus I)$
is isomorphic to $\Fl^a_{n+1}$.
We thus obtain the following corollary.
\begin{cor}
The Poincar\'e polynomial of the complete degenerate flag variety $\Fl^a_{n+1}$ is equal to
\begin{equation}\label{qG}
\sum_{f_1,\dots,f_n\ge 0}q^{\sum_{k=1}^n (k-f_k)(1-f_k+f_{k+1})}
\prod_{k=1}^{n} \binom{1+f_{k-1}}{f_k}_q
\prod_{k=1}^{n} \binom{1+f_{k+1}}{f_k}_q,
\end{equation}
(we assume $f_0=f_{n+1}=0$).
\end{cor}

Now fix a collection $\bd=(d_1,\dots,d_s)$ with $0=d_0< d_1<\dots <d_s <d_{s+1}=n+1$.
We obtain the following corollary:
\begin{cor}
Define $a_i=d_i-d_{i-1}$, $b_i=d_{i+1}-d_i$. Then formula \eqref{PP} gives the Poincar\'e
polynomial of the partial degenerate flag variety $\Fl^a_\bd$.
\end{cor}
\begin{proof}
Follows from Proposition \ref{QGCFbd}.
\end{proof}

\subsection{The normalized median Genocchi numbers}
Recall that the Euler characteristic of $\Fl^a_{n+1}$ is equal to the $(n+1)$-st normalized median
Genocchi number $h_{n+1}$ (see \cite{F2}, Proposition 3.1 and Corollary 3.7). 
In particular, the Poincar\'e polynomial \eqref{qG} provides
natural $q$-deformation $h_{n+1}(q)$. We also arrive at the following formula.
\begin{cor}
\begin{equation}\label{G}
h_{n+1}=\sum_{f_1,\dots,f_n\ge 0} \prod_{k=1}^{n} \binom{1+f_{k-1}}{f_k}
\prod_{k=1}^{n} \binom{1+f_{k+1}}{f_k}
\end{equation}
with $f_0=f_{n+1}=0$.
\end{cor}

We note that formula \eqref{G} can be seen as a sum over the set $M_{n+1}$ of Motzkin paths starting
at $(0,0)$ and ending at $(n+1,0)$.
Namely, we note that a term in \eqref{G} is zero unless $f_{i+1}=f_i$ or $f_{i+1}=f_i+1$
or $f_{i+1}=f_i-1$ for $i=1,\dots,n$ (recall that $f_i\ge 0$ and $f_0=f_{n+1}=0$).
Therefore the terms in \eqref{G} are labeled by Motzkin paths (see e.g. \cite{DS}). 
We can simplify
the expression for $h_{n+1}$. Namely, for a Motzkin path $\bff\in M_{n+1}$ let $l(\bff)$
be the number of "rises" ($f_{i+1}=f_i+1$) plus the number of "falls" ($f_{i+1}=f_i-1$). Then
we obtain
\begin{cor}
\[
h_{n+1}=\sum_{\bff\in M_{n+1}} \frac{\prod_{k=1}^n (1+f_k)^2}{2^{l(\bff)}}.
\]
\end{cor}

We note also that Remark \ref{two} produces one more combinatorial definition of the numbers $h_{n+1}$.
Namely, for $1\le i\le j\le n$ we denote by  $S_{i,j}$
the indecomposable representation of $Q$ such that
\[
{\bf dim}\, S_{i,j}=(\underbrace{0,\dots,0}_{i-1},\underbrace{1,\dots,1}_{j-i+1},0,\dots,0).
\]
In particular, the simple indecomposable representation $S_i$ coincides with $S_{i,i}$.
Then we have
\[
\dim Hom_Q(S_k,S_{i,j})=
\begin{cases}
1, \text{ if } k=j,\\
0, \text{ otherwise};
\end{cases}
\dim Hom_Q(S_{i,j},S_k)=
\begin{cases}
1, \text{ if } i=k,\\
0, \text{ otherwise}.
\end{cases}
\]
Recall (see  Theorem \ref{Thm:GroupAction}) that the Euler characteristic of $\Fl^a_{n+1}$
is equal to the number of isomorphism classes of pairs $[Q_P],[N_I]$ such that $N_I$ is embedded
into $I=\bigoplus_{k=1}^n I_k$, $Q_P$ is a quotient of $P=\bigoplus_{k=1}^n P_k$ and
${\bf dim}\, N_I={\bf dim}\, Q_P$.
Let
$$
N_I=\bigoplus_{1\le i\le j\le n} S_{i,j}^{r_{i,j}},\quad Q_P=\bigoplus_{1\le i\le j\le n} S_{i,j}^{m_{i,j}}.
$$
Then from Remark \ref{two} we obtain the following
Proposition.
\begin{prop}
The normalized median Genocchi number $h_{n+1}$ is equal to the number of pairs of collections of non-negative
integers $(r_{i,j})$, $(m_{i,j})$, $1\le i\le j\le n$ subject to the following conditions
for all $k=1,\dots,n$:
$$
\sum_{k=i}^n r_{i,k}\le 1, \quad \sum_{k=1}^j m_{k,j}\le 1,\quad
\sum_{i\le k\le j} r_{i,j}=\sum_{i\le k\le j} m_{i,j}.
$$
\end{prop}

\section{Cells and the group action in type $A$}\label{CGAA}
In this section we fix $Q$ to be the equioriented quiver of type $A_n$.
\subsection{Description of the group}
Let $P=\bigoplus_{i=1}^n P_i$ and $I=\bigoplus_{i=1}^n I_i$. As in the general case, we consider
the group
$$
G=\left[\begin{array}{cc}{\rm Aut}_Q(P)&0\\ {\rm Hom}_Q(P,I)&{\rm Aut}_Q(I)\end{array}\right],
$$
which is a subgroup of ${\rm Aut}_Q(P\oplus I)$.
\begin{rem}
The whole group of automorphisms ${\rm Aut}_Q(P\oplus I)$ is generated by $G$ and
$\exp({\rm Hom}_Q(I,P))$. We note that ${\rm Hom}_Q(I,P)$ is a one-dimensional space. In fact,
${\rm Hom}_Q(I_k,P_l)=0$ unless $k=n$, $i=1$ and $I_n\simeq P_1$. Thus $G$ "almost" coincides
with ${\rm Aut} (P\oplus I)$.
\end{rem}

We now describe $G$ explicitly.
\begin{lem}
The groups ${\rm Aut}_Q(P)$ and ${\rm Aut}_Q(I)$ are isomorphic to the Borel subgroup $B_n$ of the Lie group
$GL_n$, that is, to the group of non-degenerate upper-triangular matrices.
\end{lem}
\begin{proof}
For $g\in {\rm Aut} (P\oplus I)$ let $g_i$ be the component acting on $(P\oplus I)_i$ (at the vector
space corresponding to the $i$-th vertex). Then the map $g\mapsto g_n$ gives a group isomorphism
${\rm Aut}_Q(P)\simeq B_n$. In fact, ${\rm Hom}_Q(P_k,P_l)=0$ if $k>l$. Otherwise ($k\le l$) it is
one-dimensional
and is completely determined by the $n$-th component. Similarly,
the map $g\mapsto g_1$ gives a group isomorphism ${\rm Aut}_Q(I)\simeq B_n$.
\end{proof}

In what follows, we denote ${\rm Aut}_Q(P)$ by $B_P$ and  ${\rm Aut}_Q(I)$ by $B_I$.

\begin{prop}
The group $G$ is isomorphic to the semi-direct product $\bG_a^{n(n+1)/2}\ltimes (B_P\times B_I)$.
\end{prop}
\begin{proof}
First, the groups $B_P$ and $B_I$ commute inside $G$.
Second, the group $G$ is generated by  $B_P$, $B_I$ and $\exp({\rm Hom}_Q(P,I))$.
The group $\exp({\rm Hom}_Q(P,I))$ is abelian and isomorphic to  $\bG_a^{n(n+1)/2}$ (the
abelian version of the unipotent subgroup of the lower-triangular matrices in $SL_{n+1}$).
In fact, ${\rm Hom}_Q(P_i,I_j)$ is trivial if $i>j$ and otherwise ($i\le j$) it is one-dimensional.
Also, $\exp({\rm Hom}_Q(P,I))$ is normal in $G$.
\end{proof}

We now describe explicitly the structure of the semi-direct product.
For this we pass to the level of the Lie algebras. So let $\fb_P$ and $\fb_I$ be the Lie
algebras of $B_P$ and $B_I$, respectively ($\fb_P$ and $\fb_I$ are isomorphic to the Borel
subalgebra  of $\msl_n$).
Let $(\fn^-)^a$ be the abelian $n(n+1)/2$-dimensional Lie algebra, that is, the Lie algebra of the
group  $\bG_a^{n(n+1)/2}$. Also, let $\fb$ be the Borel subalgebra of $\msl_{n+1}$.
Recall that the degenerate Lie algebra $\msl_{n+1}^a$ is defined as $(\fn^-)^a\oplus\fb$, where
$(\fn^-)^a$ is an abelian ideal and the action of $\fb$ on $(\fn^-)^a$ is induced by the adjoint
action of $\fb$ on the quotient $(\fn^-)^a\simeq \msl_{n+1}/\fb$.
Consider the embedding $\imath_P:\fb_P\to \fb$, $E_{i,j}\mapsto E_{i,j}$ and
the embedding $\imath_I:\fb_I\to \fb$, $E_{i,j}\mapsto E_{i+1,j+1}$. These embeddings
define the structures of $\fb_P$ and $\fb_I$ modules on $(\fn^-)^a$.

\begin{prop}\label{semi}
The group $G$ is the Lie group of the Lie algebra $(\fn^-)^a\oplus\fb_P\oplus\fb_I$, where
$(\fn^-)^a$ is an abelian ideal and the structure of $\fb_P\oplus \fb_I$-module on
$(\fn^-)^a$ is defined by the embeddings $\imath_P$ and $\imath_I$.
\end{prop}
\begin{proof}
The Lie algebra of $G$ is isomorphic to the direct sum
${\rm End}_Q(P)\oplus {\rm End}_Q(I)\oplus {\rm Hom}_Q(P,I)$. Recall that the identification
${\rm Hom}_Q(P,P)\simeq \fb_P$ is given by $a\mapsto a_n$ and
the identification ${\rm Hom}_Q(I,I)\simeq \fb_I$ is given by $a\mapsto a_1$,
where $a_i$ denotes its $i$-th component for $a\in {\rm Hom}_Q(P\oplus I,P\oplus I)$.
Recall (see subsection \ref{comp}) that $(P\oplus I)_1$ is spanned by the vectors $w_{1,j}$,
$j=1,\dots,n+1$ and $w_{1,1}\in (P_1)_1$, $w_{1,j}\in (I_{j-1})_1$ for $j>1$.
Therefore, we have a natural embedding $\fb_I\subset \fb$ mapping the matrix unit $E_{i,j}$ to
$E_{i+1,j+1}$. Similarly,
$(P\oplus I)_n$ is spanned by the vectors $w_{n,j}$,
$j=1,\dots,n+1$ and $w_{n,n+1}\in (I_n)_n$, $w_{n,j}\in (P_{j})_n$ for $j<n+1$, giving
the natural embedding $\fb_I\subset \fb$, $E_{i,j}\mapsto E_{i,j}$.
With such a description it is easy to compute the commutator of an element from
$\fb_P\oplus \fb_I$ with an element from ${\rm Hom}_Q(P,I)\simeq(\fn^-)^a$.
\end{proof}

We now compare $G$ with  $SL_{n+1}^a$.
We note that the Lie algebra $\msl_{n+1}^a$ and the Lie group $SL_{n+1}^a$ have one-dimensional
centers. Namely, let $\theta$ be the longest root of $\msl_{n+1}$ and
let $e_\theta=E_{1,n}\in \fb\subset \msl_{n+1}$ be the corresponding element.
Then $e_\theta$ commutes with everything in $\msl_{n+1}^a$ and thus the exponents
$\exp(t e_\theta)\in SL_{n+1}^a$ form the center $Z$.
From Proposition \ref{semi} we obtain the following corollary.
\begin{cor}
The group $SL^a_{n+1}/Z$ is embedded into $G$.
\end{cor}

\subsection{Bruhat-type decomposition}
The goal of this subsection is to study the $G$-orbits on the degenerate flag varieties.
So let $\bd=(d_1,\dots,d_s)$ for $0=d_0<d_1<\dots <d_s<d_{s+1}=n+1$.
\begin{lem}
The group $G$ acts naturally on all degenerate flag varieties $\Fl^a_\bd$.
\end{lem}
\begin{proof}
By definition, $G$ acts on the degenerate flag variety $\Fl^a_{n+1}$. We note that there exists a map
$\Fl^a_{n+1}\to \Fl^a_\bd$ defined by $(V_1,\dots,V_n)\mapsto (V_{d_1},\dots,V_{d_s})$.
Therefore, the $G$-action on $\Fl^a_{n+1}$ induces a $G$-action on $\Fl^a_\bd$.
\end{proof}

We first work out the case $s=1$, that is, the $G$-action on the classical Grassmannian
${\rm Gr}_d(n+1)$. We first recall the cellular decomposition from \cite{F2}.
The cells are labelled by torus fixed points, that is, by
collections $L=(l_1,\dots,l_d)$ with $1\le l_1<\dots <l_d\le n+1$. The corresponding
cell is denoted by $C_L$. Explicitly, the elements of $C_L$ can be described as follows.
Let $k$ be a number such that $l_k\le d<l_{k+1}$. Recall the basis $w_1,\dots,w_{n+1}$ 
of $W=\bC^{n+1}$. We denote by $p_L\in {\rm Gr}_{d+1}(n+1)$ the linear span of $w_{l_1},\dots,w_{l_d}$.
Then
a $d$-dimensional subspace $V$ belongs to $C_L$ if and only if it has a basis $e_1,\dots,e_d$
such that for some constants $c_p$, we have
\begin{gather}
\label{cP}
e_j=w_{l_j} + \sum_{p=1}^{l_j-1} c_p w_p +\sum_{p=d+1}^{n+1} c_p w_p
\qquad \text{ for } j=1,\dots,k;\\
\label{cI}
e_j=w_{l_j} + \sum_{p=d+1}^{l_j-1} c_p w_p \qquad \text{ for } j=k+1,\dots,d.
\end{gather}
For example, $p_L\in C_L$.

\begin{lem}
Each $G$-orbit on the Grassmannian $\Fl^a_{(d)}$ contains exactly one torus fixed point $p_L$.
The orbit $G\cdot p_L$ coincides with $C_L$.
\end{lem}
\begin{proof}
Follows from the definition of $G$.
\end{proof}

We prove now that the $G$-orbits in ${\rm Gr}_{{\bf dim} P}(P\oplus I)$ described in
Theorem \ref{Thm:GroupAction} are cells. Moreover, we prove that this cellular decomposition
coincides with the one from \cite{F2}.
Let
$$P=\bigoplus_{i=1}^s P_i^{d_i-d_{i-1}},\;\;\; I=\bigoplus_{i=1}^s I_i^{d_{i+1}-d_i}.$$
We start with the following lemma.
\begin{lem}
Let $N_I\subset I$ be a subrepresentation of $I$. Then there exists a unique torus
fixed point $N^\circ_I\in {\rm Gr}_{{\bf dim} N_I}(I)$ such that $N_I\simeq N_I^\circ$.
Similarly, for $N_P\subset P$ there exists a unique torus
fixed point $N^\circ_P\in {\rm Gr}_{{\bf dim} N_P}(P)$ such that $P/N_P\simeq P/N^\circ_P$.
\end{lem}
\begin{proof}
We prove the first part, the second part can be proved similarly.
Recall the vectors  $w_{i,j}\in (I_{j-1})_i$, $i=1,\dots,n$, $j=i+1,\dots,n+1$ such that
$w_{i,j}\mapsto w_{i+1,j}$ if $j\ne i+1$ and $w_{i,j}\mapsto 0$ if $j=i+1$.
For each indecomposable summand $S_{k,l}$ of $N_I$ we construct a corresponding
indecomposable summand of $N^\circ_I$. Namely, we take the subrepresentation in $I_l$
of dimension vector
\[
(\underbrace{0,\dots,0}_{k-1},\underbrace{1,\dots,1}_{l-k+1},0,\dots,0).
\]
Since each $I_l$ is torus-fixed, our lemma is proved.
\end{proof}

\begin{rem}
This lemma is not true in general.  Namely, consider the quiver from Example \ref{D4} and let
$N_I\subset I_3\oplus I_4$ be indecomposable $Q$-module of dimension $(1,2,1,1)$.
Then  for such $N_I$ the lemma above is false.
\end{rem}

\begin{cor}
Each $G$-orbit in ${\rm Gr}_{{\bf dim} P}(P\oplus I)$ contains exactly one torus fixed point, and each
such point is contained in some orbit.
\end{cor}
\begin{proof}
Follows from Theorem \ref{Thm:GroupAction}.
\end{proof}

We note that any torus fixed point in $\Fl^a_\bd$ is the product of fixed points in
the Grassmannians $\Fl^a_{(d_i)}$, $i=1,\dots,s$. Therefore any such point is of the form
$\prod_{i=1}^s p_{L^i}$. We denote this point by $p_{L^1,\dots,L^s}$.

\begin{thm}\label{intersection}
The orbit $G\cdot p_{L^1,\dots,L^s}$ is the intersection of the quiver Grassmannian
${\rm Gr}_{{\bf dim} P}(P\oplus I)$ with the product of cells $C_{L^i}$.
\end{thm}
\begin{proof}
First, obviously $G\cdot p_{L^1,\dots,L^s}\subset \Fl^a_\bd\bigcap \prod_{i=1}^s C_{L^i}$.
Second, since each orbit contains exactly one torus fixed point and the intersection on the right hand side
does not contain other fixed points but $p_{L^1,\dots,L^s}$, the theorem is proved.
\end{proof}

\begin{cor}
The $G$-orbits on $\Fl^a_\bd$ produce the same cellular decomposition as the one constructed in \cite{F2}.
\end{cor}
\begin{proof}
The cells from \cite{F2} are labeled by collections $L^1,\dots,L^s$ (whenever
$p_{L^1,\dots,L^s}\in \Fl^a_\bd$) and the corresponding cell $C_{L^1,\dots,L^s}$ is given by
\[
C_{L^1,\dots,L^s}=\Fl^a_\bd\bigcap \prod_{i=1}^s C_{L^i}.
\]
\end{proof}

\subsection{Cells and one-dimensional torus}
In this subsection we show that the cellular decomposition described above coincides with the
one constructed in section \ref{1dtorus}.
We describe the case of the complete flag varieties (in the parabolic case everything works
in the same manner).
Recall that the action of our torus is given by the formulas
\begin{equation}\label{GT}
\la\cdot w_{i,j}=
\begin{cases}
\la^{j-2}w_{i,j},\text{ if } j>i,\\
\la^{j+n-1}w_{i,j}, \text{ if } j\le i.
\end{cases}
\end{equation}
For $n=4$ we have the following picture (compare with \eqref{wij}):
\begin{equation}\label{1d}
\xymatrix@R=6pt@C=8pt
{
\la^7&&&&&w_{4,4}&\\
\la^6&&&&w_{3,3}\ar[r]&w_{4,3}&\\
\la^5&&&w_{2,2}\ar[r]&w_{3,2}\ar[r]&w_{4,2}&\\
\la^4&&w_{1,1}\ar[r]&w_{2,1}\ar[r]&w_{3,1}\ar[r]&w_{4,1}&\\
\la^3&&w_{1,5}\ar[r]&w_{2,5}\ar[r]&w_{3,5}\ar[r]&w_{4,5}&\\
\la^2&&w_{1,4}\ar[r]&w_{2,4}\ar[r]&w_{3,4}&&\\
\la&&w_{1,3}\ar[r]&w_{2,3}&&&\\
1&&w_{1,2}&&&&\\
}
\end{equation}

\begin{prop}
The attracting $(\la\to 0)$-cell  of a fixed point $p$ of the  one-dimensional torus \eqref{GT}
coincides with the $G$-orbit $G\cdot p$.
\end{prop}
\begin{proof}
First, consider the action of our torus on each Grassmannian ${\rm Gr}_d((P\oplus I)_d)$.
Then formulas \eqref{cP} and \eqref{cI} say that the attracting cells ($\la\to 0$) coincide with
the cells $C_L$. Now Theorem \ref{intersection} implies our proposition.
\end{proof}

We note that the one-dimensional torus \eqref{GT} does not belong to $SL^a_{n+1}$
(more precisely, to the image
of $SL_{n+1}^a$ in the group of automorphisms of the degenerate flag variety).
However, it does belong to a one-dimensional extension $SL_{n+1}^a\rtimes \bC^*_{PBW}$ of the
degenerate group (see \cite{F2}, Remark 1.1). 
Recall that the extended group is the Lie group of the extended Lie algebra
$\msl_{n+1}^a\oplus\bC d_{PBW}$, where $d_{PBW}$ commutes with the generators $E_{i,j}\in\msl_{n+1}$
as follows: $[d_{PBW},E_{i,j}]=0$ if $i<j$ and $[d_{PBW},E_{i,j}]=E_{i,j}$ if $i>j$.
In particular, the action of the torus $\bC^*_{PBW}=\{\exp(\la d_{PBW}),\la\in\bC\}$ on $w_{i,j}$
is given by the formulas: $\la\cdot w_{i,j}=w_{i,j}$ if $i\ge j$ and
$\la\cdot w_{i,j}=\la w_{i,j}$ if $i<j$. For example, for $n=4$ one has the following picture
(vectors come equipped with the weights):
$$
\xymatrix@R=6pt@C=8pt
{
\la\bullet\ar[r]&\la\bullet\ar[r]&\la\bullet\ar[r]&\la\bullet&\\
\la\bullet\ar[r]&\la\bullet\ar[r]&\la\bullet&1\bullet&\\
\la\bullet\ar[r]&\la\bullet&1\bullet\ar[r]&1\bullet&\\
\la\bullet&1\bullet\ar[r]&1\bullet\ar[r]&1\bullet&\\
1\bullet\ar[r]&1\bullet\ar[r]&1\bullet\ar[r]&1\bullet&\\
}
$$
\begin{prop}
The one-dimensional torus \eqref{GT} sits inside the extended group $SL_{n+1}^a\rtimes \bC^*_{PBW}$.
\end{prop}
\begin{proof}
For any tuple of integers $k_1,\dots,k_{n+1}$ there exists a
one-dimensional torus $\bC^*_{(k_1,\dots,k_{n+1})}$ inside the Cartan subgroup of $SL_{n+1}^a$ which
acts on $w_{i,j}$ by the formula $w_{i,j}\mapsto \la^{k_j} w_{i,j}$. Direct check shows that
the torus \eqref{GT} acts as $\bC^*_{(n,n+1,\dots,2n)}\times (\bC^*_{PBW})^{-n-1}$.
\end{proof}

\section*{Acknowledgments}
This work was initiated during the authors stay at the Hausdorff Research
Institute for Mathematics during the Trimester program
"On the Interaction of Representation Theory with Geometry and Combinatorics".
The hospitality and perfect working conditions of the Institute are gratefully
acknowledged. The work of E.~F. was partially supported
by the RFBR Grant 09-01-00058,
by the grant Scientific Schools 6501.2010.2 and by the Dynasty Foundation. M.~R. would like to thank K. Bongartz for helpful discussions.

\bibliographystyle{amsplain}

\end{document}